\documentclass[a4paper,12pt]{article}

\usepackage[T1]{fontenc}
\usepackage[utf8]{inputenc}
\usepackage[english]{babel}
\usepackage[inline,shortlabels]{enumitem}
\usepackage{amsmath}
\usepackage{amssymb}
\usepackage{amsthm}
\usepackage[noadjust]{cite}
\usepackage[a4paper,total={6in,9in}]{geometry}
\usepackage[colorlinks,citecolor=blue]{hyperref}
\usepackage[capitalize,noabbrev,nameinlink]{cleveref}
\usepackage{xcolor}
\usepackage{longtable}
\usepackage{tikz,pgfplots,pgfplotstable}
\usetikzlibrary{patterns}
\usepackage{algorithmic}
\usepackage{marginnote}
\usepackage[labelformat=simple]{subcaption}

\usepackage{multirow}

\setlist{font=\normalfont,topsep=1ex,parsep=0ex}

\setlist[enumerate]{label=(\alph*)}

\numberwithin{equation}{section}
\numberwithin{table}{section}    
\numberwithin{figure}{section}

\crefname{figure}{Figure}{Figures}
\crefname{table}{Table}{Tables}
\crefname{assumption}{Assumption}{Assumptions}
\Crefname{ALC@unique}{Step}{Steps}

\newlist{alglist}{enumerate}{1}
\setlist[alglist]{topsep=1ex,parsep=0ex,leftmargin=*,label=\textbf{Step~\arabic*.}}

\DeclareMathOperator*{\argmin}{argmin}

\newtheoremstyle{bolddef}{}{}{\normalfont}{}{\bfseries}{.}{ }{\thmname{#1}\thmnumber{ #2}\thmnote{ (#3)}}

\newtheoremstyle{boldplain}{}{}{\itshape}{}{\bfseries}{.}{ }{\thmname{#1}\thmnumber{ #2}\thmnote{ (#3)}}

\theoremstyle{bolddef}
\newtheorem{definition}{Definition}[section]
\newtheorem{algorithm}[definition]{Algorithm}
\newtheorem{assumption}[definition]{Assumption}

\newtheorem{remark}[definition]{Remark}

\theoremstyle{boldplain}
\newtheorem{lemma}[definition]{Lemma}
\newtheorem{theorem}[definition]{Theorem}
\newtheorem{proposition}[definition]{Proposition}

\newlength\figureheight
\newlength\figurewidth

\pgfplotsset{width=7cm,compat=1.3}

\definecolor{todocolor}{rgb}{1.0,0.3,0.3}

\newcommand\email[1]{\href{mailto:#1}{\texttt{#1}}}

\newcommand{\orcid}[1]{ORCID: \href{https://orcid.org/#1}{#1}}

\newcommand{\mscLink}[1]{\href{http://www.ams.org/mathscinet/msc/msc2020.html?t=#1}{#1}}

\begin{document}

\title{
	\bfseries\scshape 
	Inexact Penalty Decomposition Methods for Optimization Problems with Geometric Constraints
	}

\date{\today}

\author{Christian Kanzow
	\thanks{%
		University of Würzburg,
		Institute of Mathematics,
		97074 Würzburg,
		Germany,
		\email{kanzow@mathematik.uni-wuerzburg.de},
		\orcid{0000-0003-2897-2509}
	}
	\hspace*{-4mm} \and \hspace*{-4mm}
	Matteo Lapucci
	\thanks{%
		Department of Information Engineering, Università degli Studi di Firenze, Via di Santa Marta 3,
		50139 Firenze, Italy,
		\email{matteo.lapucci@unifi.it},
		\orcid{0000-0002-2488-5486}
	} 
}

\maketitle

{
\small\textbf{\abstractname.}
This paper provides a theoretical and numerical investigation
of a penalty decomposition scheme for the solution of 
optimization problems with geometric constraints. In particular,
we consider some situations where parts of the constraints are
nonconvex and complicated, like cardinality constraints, 
disjunctive programs, or matrix problems involving rank
constraints. By a variable duplication and decomposition strategy, the method presented here explicitly handles these difficult
constraints, thus generating
iterates which are feasible with respect to them,
while the remaining (standard and supposingly 
simple) constraints are tackled by sequential penalization. Inexact optimization steps are proven sufficient for the resulting algorithm to work, so that it is employable even with difficult objective functions. 
The current work is therefore a significant generalization of
existing papers on penalty decomposition methods. On the other 
hand, it is related to some recent publications which use
an augmented Lagrangian idea to solve optimization problems
with geometric constraints. Compared to these methods, the
decomposition idea is shown to be numerically superior since
it allows much more freedom in the choice of the subproblem
solver, and since the number of certain (possibly expensive) projection steps is significantly less. Extensive numerical results 
on several highly complicated classes of optimization problems
in vector and matrix spaces indicate that the current method
is indeed very efficient to solve these problems.
\par\addvspace{\baselineskip}
}

{
\par\addvspace{\baselineskip}
}

{
	\mscLink{90C30}, \mscLink{90C33}
\par\addvspace{\baselineskip}
}

\section{Introduction}\label{Sec:Intro}

We consider the program
\begin{equation}\label{P}
	\min_x \ f(x) \quad \text{s.t.} \quad G(x) \in C, \ x \in D,
\end{equation}
where $ f: \mathbb X \to \mathbb R $ and $ G: \mathbb X \to \mathbb Y $ are continuously differentiable mappings, $ \mathbb X $ and $ \mathbb Y $ are Euclidean spaces, i.e., real and finite-dimensional Hilbert spaces, $ C \subseteq \mathbb Y $ is nonempty, closed, and convex, whereas $ D \subseteq \mathbb X $ is only assumed to be nonempty and closed (not necessarily convex), representing a possibly complicated set.

This very general setting (analyzed for example in \cite{JiaKanzowMehlitzWachsmuth2021}) covers, for example, standard nonlinear programming problems with convex constraints, but also difficult \textit{disjunctive programming} problems \cite{Benko2022,Benko2018,Flegel2007,Mehlitz2020}, e.g., complementarity \cite{Ye1999}, vanishing \cite{Achtziger2008}, switching \cite{Mehlitz2020b} and cardinality constrained \cite{Lapucci2022,Lapucci2021} problems. Matrix optimization problems such as low-rank approximation \cite{Kishore2017,Markovsky2012} are also captured by our setting.

Problems with this structure, where the feasible set consists of the intersection of a set of constraints expressed in analytical form and another complicated set, without regularity guarantees but manageable for example by easy projections, have been deeply studied in recent years.  In particular, approaches based on decomposition and sequential penalty or augmented Lagrangian methods have been proposed for the convex case \cite{Galvan2020}, the cardinality constrained case \cite{Lapucci2021,Lu2013} and the low-rank approximation case \cite{Zhang2011}; the recurrent idea in all these works consists of the application of the variable splitting technique \cite{Guignard1987,Jornsten1985}, to then define a penalty function associated with the differentiable constraints and the additional equality constraint linking the two blocks of variables and finally solve the problem by a sequential penalty method. The optimization of the penalty function is carried out by a two-block alternatimg minimization scheme \cite{Grippo2000}, which can be run in an exact \cite{Lu2013,Zhang2011} or inexact \cite{Galvan2020,Lapucci2021} fashion. 

The aim of this work is to extend the \textit{inexact Penalty Decomposition} approach to the general setting \eqref{P} in such 
a way that it can deal with arbitrary abstract constraints $ D $
(at least theoretically, in practice $ D $ needs to be such that
projections onto this set are easy to compute) and that it
allows additional (seemingly simple) constraints given by
$ G(x) \in C $. This setting is related to some recent work
on (safeguarded) augmented Lagrangian methods, see, in particular, \cite{JiaKanzowMehlitzWachsmuth2021}, where the resulting subprobems are solved by a projected gradient-type method, which might be
inefficient especially for ill-conditioned problems. 
The decomposition idea
used here allows a much wider choice of subproblem solvers, 
usually resulting in a more efficient solver of the given
optimization problem \eqref{P}. 

The paper is organized as follows: Section~\ref{Sec:Prelims}
summarizes some preliminary concepts and results. In particular,
we recall the definitions of an M-stationary point (the 
counterpart of a KKT point for the general setting from 
\eqref{P}), of an AM-stationary point (as a sequential
version of M-stationarity) and of an AM-regular point (this being
a suitable and relatively weak constraint qualification).
Section~\ref{Sec:Alg} then presents the \textit{Penalty Decomposition} method together with a global convergence theory,
assuming that the resulting subproblems can be solved inexactly
up to a certain degree. In Section~\ref{Sec:Subproblem}, we then
present a class of \textit{inexact alternating minimization} methods which, under certain assumptions, 
are guaranteed to find the desired approximate solution of the 
subproblems arising in the outer penalty scheme. The 
remaining part of the paper is then devoted to the implementation
of the overall method and corresponding numerical results. To
this end, Section~\ref{sec:realizations} first discusses 
several instances of the 
general setting \eqref{P} with difficult constraints $ D $
where our method can be applied to quite efficiently since
projections onto $ D $ are simple and/or known analytically 
(though the latter does not necessarily imply that these projections
are easy to compute numerically). In Section~\ref{sec:experiments},
we then present the results of an extensive numerical testing,
where we also compare our method, using different realizations,
with the augemented Lagrangian method from \cite{JiaKanzowMehlitzWachsmuth2021}. We conclude with some
final remarks in Section~\ref{Sec:Conclusions}.

\section{Preliminaries}\label{Sec:Prelims}

The Euclidean projection $ P_C: \mathbb Y \to \mathbb Y $ onto the nonempty, closed, and convex set $ C $ is defined by 
\begin{equation*}
	P_C (y) := \text{argmin}_{z \in C} \| z - y \| .
\end{equation*}
The corresponding distance function $ d_C : \mathbb Y \to \mathbb R $ can then be written as
\begin{equation*}
	\text{dist}_C(y) := \min_{z \in C} \| z - y \| = \| P_C (y) - y \|.
\end{equation*}
Note that the distance function is nonsmooth (in general), but the squared distance function 
\begin{equation*}
	s_C (y) := \frac{1}{2} \text{dist}_C^2(y)
\end{equation*}
is continuously differentiable everywhere with derivative given by
\begin{equation}\label{Eq:Derivative-s}
	\nabla s_C(y) = y - P_C (y),
\end{equation}
see \cite[Cor.\ 12.30]{BauschkeCombettes2011}. Moreover, projections onto the nonempty and closed set $ D $ also exist, but are not necessarily unique. Therefore, we define the (usually set-valued) projection operator $ \Pi_D : \mathbb X \rightrightarrows \mathbb X $ by 
\begin{equation*}
	\Pi_D(x) := \argmin_{z \in D} \| z - x \| \neq \emptyset .
\end{equation*}
The corresponding distance function $ \text{dist}_D ( \cdot ) $ is, of course, single-valued again.
Furthermore, given a set-valued mapping $ S: \mathbb X 
\rightrightarrows \mathbb X $ on an arbitrary Euclidean space $ \mathbb X $, we define the {\em outer limit} of $ S $ at a point $ \bar x $ by
\begin{equation*}
	\limsup_{x \to \bar x} S(x) := \big\{ y \in \mathbb X \, \big| \, \exists x^k \to \bar x, \exists y^k \to y \text{ with } y^k \in S(x^k) \ \forall k \in \mathbb{N} \big\}.
\end{equation*}
This allows to define the {\em limiting normal cone} at a point $ x \in D $ by
\begin{equation*}
	\mathcal{N}_D^{\lim} (x) := \limsup_{v \to x} \big( \text{cone} \big( v - \Pi_D (v) \big) \big),
\end{equation*}
see \cite[Sect.\ 1.1]{Mordukhovich2018} for further details. Writing
\begin{equation*}
	x \to_D \bar x \  \Longleftrightarrow \ x \to \bar x, \ x \in D
\end{equation*} 
for sequences converging to an element $ \bar x \in D $ such that the whole sequence belongs to $ D $, the limiting normal cone has the important robustness property
\begin{equation}\label{Eq:robustness}
	\limsup_{x \to_D \bar x} \mathcal{N}_D^{\lim} (x) = \mathcal{N}_D^{\lim} (\bar x)
\end{equation}
that will be exploited heavily in our subsequent analysis, see \cite[Prop.\ 1.3]{Mordukhovich2018}.

Note that, for $ D $ being convex, this limiting normal cone reduces to the standard normal cone from convex analysis, i.e., we have
\begin{equation*}
	\mathcal{N}_D^{\lim} ( \bar x ) = \mathcal{N}_D ( \bar x ) :=
    \{ \lambda \in \mathbb X \, | \, \langle \lambda, x - \bar x \rangle \leq 0 \ \forall x \in D \, \}
\end{equation*}
for any given $ \bar x \in D $. For points $ \bar x \not\in D $, we set $ \mathcal{N}_D^{\lim} ( \bar x ) : = \mathcal{N}_D ( \bar x ) := \emptyset $. For the convex set $ C $, the standard normal cone and the projection operator are related by
\begin{equation}\label{Eq:RelNProj}
	p = P_C (y) \ \Longleftrightarrow y - p \in \mathcal{N}_C (p) ,
\end{equation}
see \cite[Prop.\ 6.46]{BauschkeCombettes2011}.

We next introduce a stationarity condition which generalizes the concept of a KKT point to constrained optimization problems with possibly difficult constraints as given by the set $ D $ in our setting \eqref{P}.

\begin{definition}\label{Def:MStat}
A feasible point $ \bar x \in \mathbb X $ of the optimization problem \eqref{P} is called an {\em M-stationary point} (Mordukhovich-stationary point) of \eqref{P} if there exists a multiplier $ \lambda \in \mathbb Y $ such that
\begin{equation*}
	0 \in \nabla f(\bar x) + G'(\bar x)^* \lambda + \mathcal{N}_D^{\lim} (\bar x), \quad 
	\lambda \in \mathcal{N}_C \big( G( \bar x) \big) .
\end{equation*}
\end{definition}

\noindent
Note that this definition coincides with the one of a KKT point if $ D $ is convex. The following is a sequential version of M-stationarity.

\begin{definition}\label{Def:AMStat}
A feasible point $ \bar x \in \mathbb{X} $ of the optimization problem \eqref{P} is called an {\em AM-stationary point} (asymptotically M-stationary point) of \eqref{P} if there exist sequences $ \{ x^k \}, \{ \varepsilon^k \} \subseteq \mathbb X $ and $ \{ \lambda^k \}, \{ z^k \} \subseteq \mathbb Y $ such that $ x^k \to \bar x $, $ \varepsilon^k \to 0 $, $ z^k \to 0 $, as well as 
\begin{equation*}
	\varepsilon^k \in \nabla f(x^k) + G'(x^k)^* \lambda^k + \mathcal{N}_D^{\lim} (x^k), \quad \lambda^k \in \mathcal{N}_C \big( G(x^k) - z^k \big) 
\end{equation*}
for all $ k \in \mathbb N $.
\end{definition}

\noindent 
Note that the previous definition generalizes the related 
concept of AKKT points introduced for standard nonlinear 
programs in \cite{Andreani2011} to our setting with the 
more difficult constraints. In a similar way, the subsequent
regularity conditions are also motivated by related ones
from \cite{Andreani2016} , where they were presented for standard 
nonlinear programs. 

Every M-stationary point is obviously also AM-stationary, whereas
the opposite implication will be guaranteed to hold by a regularity
condition that will now be introduced. To this end, let us write
\begin{equation*}
	\mathcal{M} (x,z) := G'(x)^* \mathcal{N}_C \big(G(x) - z \big) +
	\mathcal{N}_D^{\lim} (x).
\end{equation*}
Recall that $ \mathcal{N}_D^{\lim} (x) $ is nonempty if and only if
$ x \in D $, which is therefore an implicit requirement for the set
$ \mathcal{M} (x,z) $ to be nonempty. Moreover, we consider the set
\begin{equation*}
	\limsup_{x \to \bar x, z \to 0} \mathcal{M} (x,z) =
	\big\{ v \, \big| \, \exists x^k \to_D \bar x, \exists z^k \to 0:
	v^k \to v \text{ and } v^k \in \mathcal{M} (x^k,z^k) \enspace
	\forall k \in \mathbb{N} \big\}.
\end{equation*}
Note that the auxiliary sequence $ \{ z^k \} $ needs to be introduced
since the elements $ G(x^k) $ do not necessarily belong to $ C $,
whereas $ x^k $ is supposed to be an element of $ D $.

\begin{definition}\label{Def:AM-regular}
Let $ \bar x $ be feasible for \eqref{P}. Then $ \bar x $ is called
{\em AM-regular} for \eqref{P} if $ \limsup_{x \to \bar x, z \to 0} \mathcal{M} (x,z) \subseteq \mathcal{M} (\bar x, 0) $.
\end{definition}

\noindent
Using this terminology, the following statements hold, cf.\ 
\cite{Mehlitz2020c} for further details.

\begin{theorem}\label{Thm:AM-Main-Theorem}
The following statements hold:
\begin{itemize}
	\item[(a)] Every local minimum of \eqref{P} is an AM-stationary 
	   point.
	\item[(b)] If $ \bar x $ is an AM-stationary point satisfying
	   AM-regularity, then $ \bar x $ is an M-stationary point of 
	   \eqref{P}.
	\item[(c)] Conversely, if for every continuously differentiable
	   function $ f $, the implication
	   \begin{center}
	   	   $ \bar x $ is an AM-stationary point $ \Longrightarrow $
	   	   $ \bar x $ is an M-stationary point
	   \end{center}
       holds for the corresponding optimization problem \eqref{P},
       then $ \bar x $ is AM-regular.
\end{itemize}
\end{theorem}

\noindent
Statement (a) shows that every local minimum of \eqref{P} is an 
AM-stationary point even in the absense of any constraint qualification
(CQ for short).
Hence AM-stationary is a (sequential) first-order optimality condition.
In order to guarantee that an AM-stationary point is already an
M-stationary point (hence a KKT point in the standard setting of
a nonlinear program, say), we require a CQ,
namely the AM-regularity condition, cf.\ Theorem~\ref{Thm:AM-Main-Theorem} (b). 
The final statement (c) of that
result shows that, in a certain sense, AM-regularity is the weakest
CQ which implies AM-stationary points to be
M-stationary. In fact, this AM-regularity condition turns out to be a fairly
weak condition. For example, for standard nonlinear programs, AM-regularity is stronger than the Abadie CQ, but weaker than most of the other well-known CQs like MFCQ (Mangasarian-Fromovitz CQ), CRCQ (constant rank CQ), CPLD (constant positive linear dependence), and RCPLD (relaxed CPLD), to mention at least some of the more prominent ones. We refer the interested reader to \cite{Andreani2018,Mehlitz2020c} and references therein
for further details. 

The algorithm, to be described in the following section, is based on the reformulation
\begin{equation}\label{P-Decomposed}
	\min_{x,y} f(x) \quad \text{s.t.} \quad x-y = 0, \ G(x) \in C, y \in D,
\end{equation}
of the given optimization problem \eqref{P}. The previous notions of M- and AM-stationarity and
AM-regularity can be directly translated to this program by observing that \eqref{P-Decomposed}
can be written in the format of \eqref{P} as
\begin{eqnarray*}
	\tilde x & := & (x,y), \\
	\tilde f ( \tilde x) & := & \tilde f (x,y) \ := \ f(x), \\
	\tilde G (\tilde x) & := & \tilde G (x,y) \ := \ \begin{pmatrix}
		G(x) \\ x - y \end{pmatrix}, \\
	\tilde C & := & C \times 0_{\mathbb X}, \\
	\tilde D & := & \mathbb X \times D .
\end{eqnarray*}
The counterpart of Theorem~\ref{Thm:AM-Main-Theorem} then also holds for the corresponding
(A)M-stationa\-rity and regularity concepts defined for the formulation~\eqref{P-Decomposed}. Note that regularity 
conditions and constraint qualifications depend on the 
explicit formulation of a constraint system, hence the 
corresponding concepts are not necessarily equivalent for 
the two formulations of our program. We stress, however, that 
an easy inspection shows that the important notion
of an M-stationary point for \eqref{P-Decomposed} is equivalent to the notion of an M-stationary point for \eqref{P}. 

For the sake of completenss, and since this condition will be used explicitly in our
convergence analysis, let us write down explicitly the resulting AM-stationarity condition
for the reformulated program \eqref{P-Decomposed}: with the above identifications, a
feasible point $ \tilde x^* $ of \eqref{P-Decomposed} is AM-stationary if there exist sequences $ \{ \tilde x^k \} $, $ \{ \tilde \varepsilon^k \} $, and $ \{ \tilde \lambda^k \} $, $  \{ \tilde z^k \} $ such that 
$ \tilde x^k \to \tilde x^* $, $ \tilde z^k \to 0 $, $ \tilde \varepsilon^k \to 0 $ as well as
\begin{equation}\label{Eq:Normal-Decomposed}
	\tilde \varepsilon^k \in \nabla \tilde f (\tilde x^k) + \tilde G' ( \tilde x^k )^* \tilde \lambda^k + \mathcal{N}_{\tilde D}^{\lim} ( \tilde x^k ) \quad \text{and} \quad \tilde \lambda^k \in \mathcal{N}_{\tilde C} \big( \tilde G ( \tilde x^k) - \tilde z^k \big)
\end{equation}
for all $ k $. Using the definitions of $ \tilde f, \tilde G $ etc., exploiting standard properties of the limiting and standard normal cones (in particular, the Cartesian product rule, cf.\ \cite[Prop.\ 6.41]{RockafellarWets2009}), and writing $ \tilde x^k =: (x^k,y^k), \tilde \lambda^k =: ( \lambda^k, \mu^k) $ as well as $ z^k $ for the first block of $ \tilde z^k $ (the second block component of $ \tilde z^k $
turns out to be irrelevant), we see that the two conditions from \eqref{Eq:Normal-Decomposed} can be rewritten as
\begin{equation}\label{Eq:TwoInclusions}
	\tilde \varepsilon^k \in \begin{pmatrix}
		\nabla f(x^k) + G'(x^k)^* \lambda^k + \mu^k \\
		- \mu^k + \mathcal{N}_D^{\lim} (y^k) 
	\end{pmatrix} \quad \text{and} \quad 
	\begin{pmatrix}
		\lambda^k \\ \mu^k 
	\end{pmatrix}
	\in 
	\begin{pmatrix}
		\mathcal{N}_C \big( G(x^k) - z^k \big) \\ \mathbb X
	\end{pmatrix} .
\end{equation}

\section{Algorithm and Convergence}\label{Sec:Alg}

The algorithm to be presented here is based on the reformulation \eqref{P-Decomposed} of the
given program \eqref{P}. The idea is to take advantage of the fact that the constraints
$ G(x) \in C $ and $ y \in D $ occur in a decomposed way. This formulation allows to develop an alternating
direction-type penalty scheme for the solution of the orginal problem \eqref{P}. To this end, let $ \tau > 0 $ be a penalty parameter and define the partial penalty function
\begin{equation}
	q_{\tau} (x,y) := f(x) + \frac{\tau}{2} \Big( \| x-y \|^2 + \text{dist}_C^2 \big( G(x) \big) \Big).
\end{equation}
Note that $ q_{\tau} $ does not include the potentially difficult constraint $ x \in D $, which we therefore have to deal with explicitly. The general algorithmic scheme that we will investigate here is the following one.

\begin{algorithm}\label{Alg:PenDecomp}
(Inexact Penalty Decomposition Method)
\begin{itemize}
	\item[(S.0)] Choose $ \delta_0 \geq 0 $ and $ \tau_0 > 0 $, a starting point $ (x^0,y^0) \in \mathbb X \times \mathbb X $, and set $ k := 0 $.
	\item[(S.1)] If a suitable termination criterion holds: STOP.
	\item[(S.2)] Compute $ \big( x^{k+1}, y^{k+1} \big) $ such that 
	   \begin{equation}\label{Eq:Innerk-1}
	   	   \big\| \nabla_x q_{\tau_k} (x^{k+1}, y^{k+1} ) \big\| \leq \delta_k
	   \end{equation}
    and
    \begin{equation}\label{Eq:Innerk-2}
    	y^{k+1} \in \text{argmin}_{y \in D} \ q_{\tau_k} (x^{k+1},y)
    \end{equation}
    hold.
    \item[(S.3)] Choose $ \delta_{k+1} \leq \delta_k, \tau_{k+1} > \tau_k, k \leftarrow k +1 $, and go to (S.1).
\end{itemize}
\end{algorithm}

\noindent
Note that Algorithm~\ref{Alg:PenDecomp} is a very general scheme for the solution of the reformulated problem \eqref{P-Decomposed}. The main computational burden is in (S.2). We will see how this step can be realized by an alternating minimization-type iteration
in Section~\ref{Sec:Subproblem}. Here we only note that the computation of the exact minimizer $ y^{k+1} = \text{argmin}_{y \in D} \ q_{\tau_k} (x^{k+1},y) $ can be carried out very easily if projections onto the 
set $ D $ can be computed efficiently. This follows immediately from the definition of $ q_{\tau_k} $, which implies that $ y^{k+1} $ is characterized by
\begin{equation}
	\label{y-proj}
	y^{k+1} \in \Pi_D (x^{k+1}).
\end{equation}
Some examples of complicated (nonconvex) sets $ D $, where this projection is easy to compute, will be given in the numerical section.

The remaining part of this section is devoted to the global
convergence properties of the general scheme from Algorithm~\ref{Alg:PenDecomp}. The technique of proof patterns the one used in \cite{JiaKanzowMehlitzWachsmuth2021} for an augmented Lagrangian method. 

We begin with a feasibility-type result. To this end, recall that all penalty-type methods suffer from the fact that accumulation points may not be feasible for the given optimization problem. The following result shows that such an accumulation point still has a very nice property.

\begin{proposition}\label{Prop:Feasibility}
Let $ \{ (x^k,y^k) \} $ be a sequence generated by Algorithm~\ref{Alg:PenDecomp} with $ \{ \delta_k \} $ being bounded and $ \{ \tau_k \} \to \infty $. Then every accumulation point $ (\bar x, \bar y) $ of the sequence $ \{ (x^k,y^k) \} $ is an M-stationary point of the feasibility problem
\begin{equation}\label{Eq:FeasProblem}
   \min_{x,y} \ \frac{1}{2} \text{dist}_C^2 \big( G(x) \big) + \frac{1}{2} \| x-y \|^2
   \quad \text{s.t.} \quad y \in D.
\end{equation}
\end{proposition}

\begin{proof}
Let $ \{ (x^{k+1}, y^{k+1}) \}_K $ be a subsequence converging to $ (\bar x, \bar y) $. Using the derivative formula of the distance function from \eqref{Eq:Derivative-s} together with the chain rule, we have, by construction,
\begin{eqnarray*}
	\lefteqn{\big\| \nabla_x q_{\tau_k} (x^{k+1}, y^{k+1}) \big\|} \\
    & = & \Big\| \nabla f(x^{k+1}) + \tau_k \Big[ G'(x^{k+1})^* \big[ G(x^{k+1}) - P_C \big( G(x^{k+1}) \big) \big] + x^{k+1} - y^{k+1} \Big] \Big\| \\
    & \leq & \delta_k
\end{eqnarray*}
and 
\begin{equation*}
	0 \in \nabla_y q_{\tau_k} \big(x^{k+1},y^{k+1}\big) + \mathcal{N}_D^{\lim} (y^{k+1})  = \tau_k \big( y^{k+1} - x^{k+1} \big) + \mathcal{N}_D^{\lim} (y^{k+1})
\end{equation*}
for all $ k \in \mathbb N $. Dividing the first equation by $ \tau_k $ and exploiting the cone property in the second inclusion yields 
\begin{equation*}
	\Big\| \frac{1}{\tau_k} \nabla f(x^{k+1}) + G'(x^{k+1})^* \big[ G(x^{k+1}) - P_C \big( G(x^{k+1}) \big) \big] + x^{k+1} - y^{k+1} \Big\| \leq \frac{\delta_k}{\tau_k}
\end{equation*}
and 
\begin{equation*}
	0 \in y^{k+1} - x^{k+1} + \mathcal{N}_D^{\lim} (y^{k+1})
\end{equation*}
for all $ k \in \mathbb{N} $, cf.\ \cite[Thm.\ 6.12]{RockafellarWets2009}. Taking the limit $ k \to_K \infty $, using the continuity of $ \nabla f, G, G', P_C $, and the robustness property \eqref{Eq:robustness} of the limiting normal cone, we obtain
\begin{equation*}
	G'( \bar x)^* \big( G( \bar x ) - P_C \big( G( \bar x) \big)  \big) + \bar x - \bar y = 0
	\quad \text{and} \quad 
	0 \in \bar y - \bar x + \mathcal{N}_D^{\lim} (\bar y).
\end{equation*}
This shows that $ ( \bar x, \bar y ) $ is an M-stationary point of \eqref{Eq:FeasProblem}.
\end{proof}

\noindent 
Recall that Algorithm~\ref{Alg:PenDecomp} automatically generates iterates $ y^k $ which belong to the set $ D $. The objective function in \eqref{Eq:FeasProblem} therefore only measures the violation of the constraints $ G(x) \in C $ and $ x - y = 0 $, which is included in the penalty term of $ q_{\tau} $, i.e., \eqref{Eq:FeasProblem} is a feasibility problem of the decomposed problem \eqref{P-Decomposed}. If $ \bar x = \bar y $, then $ \bar x $ turns out to be an M-stationary point of
\begin{equation*}
	\min_x \ \frac{1}{2} \text{dist}_C^2 \big( G(x) \big) 
	\quad \text{s.t.} \quad x \in D,
\end{equation*}
which is the feasibility problem of the original problem \eqref{P}. Though Proposition~\ref{Prop:Feasibility} obviously does not guarantee that an accumulation point is feasible (either for the original or the decomposed formulation), it guarantees at least a stationarity property, which is the best one can expect in general. Moreover, if the feasible set of \eqref{P} is nonempty and the function $ \frac{1}{2} \text{dist}_C^2 \big( G(x) \big) $ of \eqref{Eq:FeasProblem} is convex, then every M-stationary point is a global minimum and, hence, a feasible point of \eqref{P} or \eqref{P-Decomposed}. Note that the square of the above distance function is automatically convex if the constraint $ G(x) \in C $ satisfies standard conditions which imply that this set is convex by itself.

Moreover, one can also define an extended Robinson-type constraint qualification (which boils down to the extended MFCQ condition for standard nonlinear programs) which automatically imply that accumulation points of a sequence generated by Algorithm~\ref{Alg:PenDecomp} are feasible, cf.\ \cite{Boergens2019,Kanzow2017} for further details.

Hence, under reasonable assumptions, we can guarantee that accumulation points are automatically feasible for \eqref{P} or \eqref{P-Decomposed}, whereas, in general, they are at least M-stationary points. The following global convergence result therefore assumes that we have a feasible accumulation point and shows that this one is automatically AM-stationary for problem \eqref{P-Decomposed}. 

\begin{theorem}\label{Thm:Convergence}
Let $ \{ (x^k,y^k) \} $ be a sequence generated by Algorithm~\ref{Alg:PenDecomp} with $ \{ \delta_k \} \to 0 $ and $ \{ \tau_k \} \to \infty $, and let $ (\bar x , \bar y) $ be a feasible accumulation point of this sequence. Then $ ( \bar x, \bar y) $ is an AM-stationary point of the optimization problem \eqref{P-Decomposed}.
\end{theorem}

\begin{proof}
Let $ \{ (x^{k+1}, y^{k+1}) \}_K $ be a subsequence converging to $ (\bar x, \bar y) $. Recall that $ ( \bar x, \bar y ) $ is feasible, hence $ G( \bar x ) \in C $ and $ \bar x = \bar y \in D $. We further define the sequences 
\begin{eqnarray*}
	z^{k+1} & := & G (x^{k+1}) - P_C \big( G(x^{k+1}) \big), \\
	\lambda^{k+1} & := & \tau_k \big( G (x^{k+1}) - P_C  \big( G(x^{k+1}) \big)  \big), \\
	\mu^{k+1} & := & \tau_k \big( x^{k+1} - y^{k+1} \big), \\
	\varepsilon^{k+1} & := & \nabla f(x^{k+1}) + \tau_k \big[ G'(x^{k+1})^* \big( G(x^{k+1}) - P_C (G(x^{k+1})) \big) + x^{k+1} - y^{k+1} \big].
\end{eqnarray*}
Then setting 
\begin{equation*}
	\tilde x^{k+1} := \begin{pmatrix} 
	x^{k+1} \\ y^{k+1}
    \end{pmatrix}, \quad
    \tilde \varepsilon^{k+1} := \begin{pmatrix}
    \varepsilon^{k+1} \\ 0
    \end{pmatrix}, \quad
    \tilde \lambda^{k+1} := \begin{pmatrix}
    \lambda^{k+1} \\ \mu^{k+1}
    \end{pmatrix}, \quad 
    \tilde z^{k+1} := \begin{pmatrix}
    z^{k+1} \\ 0
    \end{pmatrix},
\end{equation*}
we claim that the corresponding four (sub-) sequences $ \{ \tilde x^{k+1} \}_K = \{ (x^{k+1}, y^{k+1}) \}_K $, $ \{ \tilde z^{k+1} \} = \{ ( z^{k+1}, 0 ) \}_K $, $ \{ \tilde \varepsilon^{k+1} \}_K = \{ ( \varepsilon^{k+1}, 0 ) \}_K $, and $ \{ \tilde \lambda^{k+1} \}_K = \{ ( \lambda^{k+1}, \mu^{k+1} ) \}_K $ satisfy the properties of an AM-stationary point for problem \eqref{P-Decomposed} as stated at the end
of Section~\ref{Sec:Prelims}, cf.\ \eqref{Eq:Normal-Decomposed} and \eqref{Eq:TwoInclusions}. First of all, $ (\bar x, \bar y) $ is feasible and $ \{ ( x^{k+1}, y^{k+1} ) \}_K \to ( \bar x, \bar y ) $ by assumption. Furthermore, by definition of $ \varepsilon^{k+1} $ and the construction of Algorithm~\ref{Alg:PenDecomp}, we also have 
\begin{equation*}
	\| \varepsilon^{k+1} \| = \| \nabla_x q_{\tau_k} (x^{k+1}, y^{k+1}) \| \leq \delta_{k} \to 0 .
\end{equation*}
This obviously implies $ \| \tilde \varepsilon^{k+1} \| \to 0 $. Furthermore, the definitions of $ \lambda^{k+1} $ and $ \mu^{k+1} $ together with $ 0 \in \tau_k \big( y^{k+1} - x^{k+1} \big) + \mathcal{N}_D^{\lim} (y^{k+1}) $ yield
\begin{equation*}
	\varepsilon^{k+1} = \nabla f(x^{k+1}) + G'(x^{k+1})^* \lambda^{k+1} + \mu^{k+1} \quad \text{and} \quad 0 \in - \mu^{k+1} + \mathcal{N}_D^{\lim} (y^{k+1}),
\end{equation*}
hence the first inclusion in \eqref{Eq:TwoInclusions} holds. To verify the second inclusion, we only have to take a closer look at the first block. By definition of $ \lambda^{k+1} $ and the relation \eqref{Eq:RelNProj} between the projection and the normal cone, we get
\begin{equation*}
	\lambda^{k+1} = \tau_k \big( G (x^{k+1}) - P_C  \big( G(x^{k+1}) \big)  \big) \in \mathcal{N}_C \big( P_C (G(x^{k+1}))\big) =  \mathcal{N}_C \big( G(x^{k+1}) - z^{k+1} \big),
\end{equation*}
where the last identity comes from the definition of $ z^{k+1} $. Finally, we also have $ \tilde z^{k+1} \to_K 0 $ since $ z^{k+1} $ satisfies
\begin{equation*}
	z^{k+1} = G (x^{k+1}) - P_C \big( G(x^{k+1}) \big) \to_K G(\bar x) - P_C \big( G (\bar x) \big) = 0
\end{equation*}
by the continuity of $ G $ and the projection operator $ P_C $ as well as the feasibility of $ \bar x $. Altogether, this shows that $ (\bar x, \bar y) $ is an AM-stationary point of the program \eqref{P-Decomposed}.
\end{proof}

\noindent
Using Theorem~\ref{Thm:Convergence} together with the counterpart of Theorem~\ref{Thm:AM-Main-Theorem}
for \eqref{P-Decomposed}
and the fact that the M-stationarity conditions for the two problems \eqref{P} and \eqref{P-Decomposed}
are equivalent, we directly obtain the following result.

\begin{theorem}\label{Thm:GlobConvMStat}
Let $ \{ (x^k,y^k) \} $ be a sequence generated by Algorithm~\ref{Alg:PenDecomp} with $ \{ \delta_k \} \to 0 $ and $ \{ \tau_k \} \to \infty $, and let $ (\bar x , \bar y) $ be a feasible accumulation point of this sequence satisfying AM-regularity for \eqref{P-Decomposed}. Then $ ( \bar x, \bar y) $ is an M-stationary point of the optimization problem \eqref{P-Decomposed}, and $ \bar x $ itself is an M-stationary point of
the original problem \eqref{P}.
\end{theorem}

\section{Solution of Subproblems by Inexact Alternating
	 Minimization}\label{Sec:Subproblem}

The Penalty Decomposition approach basically consists in approximately solving the
sequence of penalty subproblems at step (S.2) by a two-block decomposition method. The alternating minimization loop can be stopped, at each iteration, as soon as an approximate stationary point of the penalty function w.r.t.\ the first block of variables $x$ is attained. The instructions of the (inexact) Alternating Minimization loop at a fixed iteration $k$ of the Penalty Decomposition method are detailed in Algorithm \ref{Alg:AM}. 

\begin{algorithm}\label{Alg:AM}
	(Inexact Alternating Minimization)
	\begin{itemize}
		\item[(S.0)] Given $ \delta_k \geq 0 $ and $ \tau_k > 0 $, a starting point $ (x^k,y^k) \in \mathbb X \times D $, $\gamma \in (0,1)$, $\beta \in (0,1)$,  set $\ell := 0$, $(u^0,v^0) = (x^{k},y^{k})$.
		\item[(S.1)] If $\big\| \nabla_x q_{\tau_k} (u^{\ell}, v^{\ell} ) \big\| \leq \delta_k$: STOP returning $(x^{k+1},y^{k+1}) = (u^\ell,v^\ell)$.
		\item[(S.2)] Choose a positive definite self-adjoint linear map $H_\ell$, set $d^\ell = -H_\ell(\nabla_x q_{\tau_k}(u^\ell,v^\ell))$ and compute \begin{equation}
			\label{eq:arm_acc}
			\alpha_\ell = \max_{j\in\mathbb{N}} \{ \beta^j: q_{\tau_k}(u^\ell+ \beta^j d^\ell, v^\ell ) \leq 
			q_{\tau_k}(u^\ell, v^\ell ) + \gamma\beta^j \langle \nabla_x q_{\tau_k}(u^\ell,v^\ell),d^\ell\rangle\}
		\end{equation}
		\item[(S.3)] Set $u^{\ell+1} = u^\ell +\alpha_\ell d_\ell$.
		\item[(S.4)] Compute $v^{\ell+1}\in \argmin_{v\in D}q_{\tau_k}(u^{\ell+1}, v)=\Pi_D(u^{\ell+1})$.
		\item[(S.5)] Set $\ell= \ell+1$ and go to (S.1).
	\end{itemize}
\end{algorithm} 

\noindent
As already pointed out, if we assume that projections onto the set $D$ are easily computable, the update of the second block of variables can be carried out exactly by \eqref{y-proj}. 

On the other hand, an exact $x$-update step may be prohibitive in most applications. For this reason, the $x$-variable is only updated by a descent step along a descent direction, with a step size selected by an Armijo-type line search. 

Note that the direction $d_\ell= -H_\ell(\nabla_x q_{\tau_k}(u^\ell,v^\ell))$ is certainly a descent direction, since $H_\ell$ is positive definite, $\nabla_xq_{\tau_k}(u^\ell,v^\ell)\neq 0$ and, thus,
\begin{equation}
	\label{eq:descent_dir}
	\langle \nabla_xq_{\tau_k}(u^\ell,v^\ell),d^\ell\rangle=-\langle \nabla_xq_{\tau_k}(u^\ell,v^\ell), H_\ell (\nabla_xq_{\tau_k}(u^\ell,v^\ell)) \rangle<0.
\end{equation} 

The Armijo line search provides a sufficient decrease granting, under suitable assumptions on the sequence of maps $H_\ell$, the convergence of the entire alternate minimization scheme.

Note that by properly choosing $H_\ell$ we can retrieve the descent directions employed in most widely employed nonlinear optimization solvers. This point, which we will emphasize again later on, is particularly relevant from the computational point of view. 
%

Throughout this section, we make the following assumption.

\begin{assumption}
	\label{ass:bounded}
	$f(x)$ has bounded level sets upon $\mathbb{X}$, i.e., $\mathcal{L}_f(\eta) = \{x\in \mathbb{X}\mid f(x)\le \eta\}$ is bounded for any $\eta\in\mathbb{R}$.
\end{assumption}

We then begin by proving that, under Assumption \ref{ass:bounded}, the penalty function has bounded level sets for any nonnegative value of the penalty parameter $\tau$. 

\begin{lemma}
	\label{lemma:bounded}
	The penalty function $q_{\tau}(x,y)$ has bounded level sets for any $\tau\ge 0$.
\end{lemma}

\begin{proof}
	Consider any $\eta\in\mathbb{R}$. From Assumption \ref{ass:bounded}, the level set $\mathcal{L}_f(\eta)$ is bounded. Let us consider $\mathcal{L}_{q_\tau}(\eta)$ for any $\tau\ge0$. 
	
	Assume by contradiction that $\mathcal{L}_{q_\tau}(\eta)$ is not bounded, i.e., there exists $\{(x^t,y^t)\}$ such that $(x^t,y^t)\in\mathcal{L}_{q_\tau}(\eta)$ for all $t$ and $\|(x^t,y^t)\|\to \infty$. Then, either $\|x^t\|\to \infty$ or $\|y^t\|\to \infty$. 
	
	If $\|x^t\|\to \infty$, we have $f(x^t)>\eta$ for $t$ sufficiently large, being $\mathcal{L}_f(\eta)$ bounded. But then, from the definition of $q_{\tau}(x,y)$, we have for $t$ sufficiently large $q_{\tau}(x^{t},y^t)\ge f(x^t)>\eta$, which contradicts $\{(x^t,y^t)\}\subseteq\mathcal{L}_{q_\tau}(\eta)$.
	
	Thus, $\|y^t\|\to \infty$ while $\|x^t\|$ stays bounded. However, $$
	   q_{\tau}(x^t,y^t) = f(x^t) + \frac{\tau}{2}\left(\|x^t-y^t\|^2 + \text{dist}_C^2 \big( G(x) \big) \right)>\eta
	$$ 
	for $t$ sufficiently large, as $\|x^t-y^t\|^2\to \infty $, $ \text{dist}_C^2 \big( G(x) \big) \geq 0 $ and $f$ is bounded having compact level sets. This again is a contradiction, which completes the proof.
\end{proof}

\noindent
It can be easily seen that step (S.2) of Algorithm \ref{Alg:AM} is well-defined, i.e.,
there exists a finite integer $j$ such that $\beta^j$ satisfies the acceptability condition \eqref{eq:arm_acc}. Moreover the following result can be readily obtained by standard results in nonlinear optimization \cite{Bertsekas1997}.

\begin{lemma}
	\label{lemma:armijo}
	Let $\{ ( u^\ell,v^\ell ) \}$ be the sequence generated by Algorithm \ref{Alg:AM}. Let $T\subseteq \{0,1,2,\ldots \}$ be an infinite subset such that
	$$\lim\limits_{\substack{\ell\to\infty\\\ell\in T}}(u^\ell,v^\ell) = (\bar u,\bar v).$$ Let $\{d^\ell\}$ be a sequence of directions such that $\langle \nabla_x q_{\tau_k}(u^\ell,v^\ell),d^\ell\rangle<0$ and assume that $\|d^\ell\|\le M$  for some $M>0$ and for all $\ell\in T$. If, for any fixed (outer iteration) $ k $,  the following equation holds
	\begin{equation*}
	\lim\limits_{\substack{\ell \to \infty\\\ell\in T}} q_{\tau_k}(u^\ell, v^\ell)-q_{\tau_k}(u^\ell+ \alpha_\ell d^\ell, v^\ell )= 0,
	\end{equation*}
	then we have
	\begin{equation*}
	\lim\limits_{\substack{\ell \to \infty\\\ell\in T}} \langle \nabla_x q_{\tau_k}(u^\ell,v^\ell),d^\ell\rangle = 0.
	\end{equation*}
\end{lemma}

\begin{proof}
	Since, for any $\ell$, $\alpha_\ell$ is chosen according to \eqref{eq:arm_acc}, we have
	\begin{equation*}
	q_{\tau_k}(u^{\ell+1},v^\ell) \le q_{\tau_k}(u^\ell,v^\ell) + \gamma \alpha_\ell \langle \nabla_x q_{\tau_k}(u^\ell,v^\ell),d^\ell\rangle.
	\end{equation*}
	Taking the limits for $\ell\in T$, $\ell\to \infty$,
	we get
	$$\lim\limits_{\substack{\ell \to \infty\\\ell\in T}} q_{\tau_k}(u^\ell+ \alpha_\ell d^\ell, v^\ell ) - q_{\tau_k}(u^\ell, v^\ell) \le \lim_{\substack{\ell\to\infty\\\ell\in T}}\gamma \alpha_\ell\langle \nabla_x q_{\tau_k}(u^\ell,v^\ell),d^\ell\rangle\le 0,$$
	where the last inequality comes from the fact that $\gamma>0$, $\alpha_\ell\ge 0$ and $\langle \nabla_x q_{\tau_k}(u^\ell,v^\ell),d^\ell\rangle<0$ by assumption.
	From the hypotheses, we also have that the leftmost limit goes to 0, hence we obtain
	\begin{equation}\label{Eq:ProdLimit}
		\lim_{\substack{\ell\to\infty\\\ell\in T}}\gamma \alpha_\ell\langle \nabla_x q_{\tau_k}(u^\ell,v^\ell),d^\ell\rangle= 0.
	\end{equation}
	Assume, by contradiction, that $\langle \nabla_x q_{\tau_k}(u^\ell,v^\ell),d^\ell\rangle_T $ does not 
	converge to zero. Subsequencing if necessary, we may
 	assume that $ \lim_{\ell\to_T \infty} \langle \nabla_x q_{\tau_k}(u^\ell,v^\ell),d^\ell\rangle = - \nu $
 	for some number $ \nu > 0 $. On the other hand, we have
 	$ (u^\ell,v^\ell)\to_T(\bar{u},\bar{v})$ by assumption,
 	and $ \{ d^{\ell} \} $ is bounded, so we may also assume 
 	that $ \{ d^{\ell} \}_T \to \bar d $ for some limit point
 	$ \bar d $. Altogether, we then have 
	$$
		\langle \nabla_x q_{\tau_k}(\bar{u},\bar{v}),\bar{d}\rangle=\lim_{\substack{\ell\to\infty\\\ell\in T}} \langle \nabla_x q_{\tau_k}(u^\ell,v^\ell),d^\ell\rangle = - \nu <0.
	$$
	Exploiting \eqref{Eq:ProdLimit}, we see that
	$\alpha_\ell\to_{T} 0$ holds. Consequently, for all $\ell\in T$ sufficiently large, we have $\alpha_\ell<\beta^0=1$ and thus
	$$
		q_{\tau_k}\left(u^\ell+\frac{\alpha_\ell}{\beta}d^\ell,v^\ell\right)>q_{\tau_k}(u^\ell,v^\ell)+\gamma \frac{\alpha_\ell}{\beta}\langle\nabla_x q_{\tau_k}(u^\ell,v^\ell),d^\ell\rangle.
	$$
	By the mean-value theorem, we can write
	$$
		q_{\tau_k}\left(u^\ell+\frac{\alpha_\ell}{\beta}d^\ell,v^\ell\right) = q_{\tau_k}(u^\ell,v^\ell) + \frac{\alpha_\ell}{\beta}\langle \nabla_x q_{\tau_k}(z^\ell,v^\ell) ,d^\ell\rangle
	$$
	for some $z^\ell=u^\ell+\theta_\ell \frac{\alpha_\ell}{\beta} d^\ell$, $\theta_\ell\in(0,1)$. Subtracting the last two relations and dividing by $\alpha_\ell/\beta$, we get
	$$0 > \gamma\langle \nabla_x q_{\tau_k}(u^\ell,v^\ell) ,d^\ell \rangle -\langle\nabla_x q_{\tau_k}(z^\ell,v^\ell) ,d^\ell\rangle.$$
	On the other hand, 
	$$
		\lim_{\substack{\ell\in T\\\ell\to \infty}}z^\ell = \lim_{\substack{\ell\in T\\\ell\to \infty}}u^\ell + \theta_\ell\frac{\alpha_\ell}{\beta}d^\ell = \bar{u}
	$$
	since $\alpha_\ell\to_{T}0$ and $d^\ell\to \bar{d}$.	
	Taking the limits in the previous inequality, we finally get
	$$\gamma \langle \nabla_x q_{\tau_k}(\bar{u},\bar{v}), \bar{d} \rangle\le \langle \nabla_x q_{\tau_k}(\bar{u},\bar{v}), \bar{d} \rangle,$$
	which is absurd since $\gamma \in (0,1)$ and $\langle \nabla_x q_{\tau_k}(\bar{u},\bar{v}), \bar{d} \rangle=-\nu<0$. 
\end{proof}

\noindent
In order to ensure that the sequence generated by the Alternating Minimization scheme properly converges, we need the sequence of directions $\{d_\ell\}$ to satisfy suitable properties. Here, in particular, we assume that the entire sequence of linear mappings $\{H_\ell\}$ satisfies the \textit{bounded eigenvalues} condition \cite[Sec.\ 1.2]{Bertsekas1997}:
\begin{equation}
	\label{bounded_eig}
	c_1\|z\|^2\le\langle z,H_\ell (z)\rangle\le c_2\|z\|^2\qquad \forall \,z\in \mathbb{X}.
\end{equation}
We are finally able to show that the inexact alternating minimization loop stops in a finite number of iterations providing a point $(x^{k+1},y^{k+1})$ satisfying conditions \eqref{Eq:Innerk-1}-\eqref{Eq:Innerk-2}. 

\begin{proposition}\label{wdef}
	Assume the sequence of linear maps $\{H_\ell\}$ in Algorithm \ref{Alg:AM} satisfies the bounded eigenvalues condition \eqref{bounded_eig}. Then the algorithm cannot cycle infinitely and determines in a finite number of iterations a point $(x^{k+1},y^{k+1})$ such that
	\begin{equation*}
	\big\| \nabla_x q_{\tau_k} (x^{k+1}, y^{k+1} ) \big\| \leq \delta_k
	\end{equation*}
	and
	\begin{equation*}
	y^{k+1} = \argmin_{y \in D} \ q_{\tau_k} (x^{k+1},y).
	\end{equation*}
\end{proposition}

\begin{proof}
	Suppose, by contradiction that, for some values of $\tau_k$ and $\delta_k$, the sequence $\{ ( u^\ell, v^\ell ) \}$ is infinite. From the instructions of the algorithm, it is easy to see that
	we have
	\begin{equation*}
	q_{\tau_k}(u^{\ell+1}, v^{\ell+1}) \le q_{\tau_k}(u^0, v^0),
	\end{equation*}
	cf.\ \eqref{a1}.
	Hence, for all $\ell\ge 0$, the point $ ( u^\ell, v^\ell ) $
	belongs to the level set
	$$
	\{(u,v)\in \mathbb{X}\times \mathbb{X}\mid  \ q_{\tau_k}(u, v) \le q_{\tau_k}(u^0, v^0)\}.
	$$
	Lemma \ref{lemma:bounded} implies that this is a bounded set.
	Therefore, the sequence $\{ ( u^\ell, v^\ell ) \}$ admits cluster points.
	Let $K \subseteq \mathbb{N} $ be an infinite subset such that
	\begin{equation*}
	\lim_{\substack{\ell\to\infty\\\ell\in K}} (u^\ell, v^\ell) = (\bar{u}, \bar{v}).	
	\end{equation*}
	Recalling the continuity of the gradient, we have
	$$\lim_{\substack{\ell\to\infty\\\ell\in K}}\nabla_x q_{\tau_k}(u^\ell,v^\ell) = \nabla_x q_{\tau_k}(\bar{u},\bar{v}).$$
	We now show that
	$
	\nabla_x q_{\tau_k}(\bar{u},\bar{v})=0.
	$
	Taking
	into account  the instructions of the algorithm, we have
	\begin{equation}\label{a1}
	q_{\tau_k}(u^{\ell+1}, v^{\ell+1}) \le q_{\tau_k}(u^{\ell+1}, v^{\ell}) = q_{\tau_k}(u^{\ell} +\alpha_\ell d^\ell, v^\ell) < q_{\tau_k}(u^{\ell}, v^{\ell}).
	\end{equation}
	By \eqref{bounded_eig}, it can be easily seen that
	$$
		\|d^\ell\|^2\le c_2^2\|\nabla_x q_{\tau_k}(u^{\ell}, v^{\ell})\|^2.
	$$
	Since
	$\nabla_x q_{\tau_k}(u^{\ell}, v^{\ell}) \to_K \nabla_x q_{\tau_k}(\bar{u},\bar{v})$,
	we see that there exists a constant $ M > 0 $ such that
	$\|d^\ell\|\le M$ for all $\ell\in K$.

	Since the entire sequence $\{q_{\tau_k}(u^{\ell}, v^{\ell})\}$ is monotonically decreasing by \eqref{a1}, and the subsequence $\{ q_{\tau_k}(u,v) \}_K$ converges to $ q_{\tau_k}(\bar{u}, \bar{v}) $, it follows that the whole sequence of function values converges
	to this limit, i.e., we have
	$$
		\lim_{\ell\to\infty}{q_{\tau_k}(u^{\ell}, v^{\ell}) = q_{\tau_k}(\bar{u}, \bar{v})}.
	$$	
	Hence \eqref{a1} yields $\lim\limits_{\ell \to \infty} q_{\tau_k}(u^{\ell}, v^{\ell}) - q_{\tau_k}(u^{\ell} + \alpha_\ell d^\ell, v^\ell) = 0.$
	Thus, the hypotheses of Lemma~\ref{lemma:armijo} are satisfied. Moreover, from \eqref{eq:descent_dir} and \eqref{bounded_eig}, we have
	$$
		\langle \nabla_x q_{\tau_k}(u^\ell,v^\ell),d^\ell\rangle\le -c_1\|\nabla_x q_{\tau_k}(u^\ell,v^\ell)\|^2.
	$$
	Using Lemma~\ref{lemma:armijo}, we therefore obain
	\begin{equation*}
	0 = \lim_{\substack{ \ell \to \infty\\\ell \in K}} \langle \nabla_x q_{\tau_k}(u^\ell,v^\ell),d^\ell\rangle \le
	\lim_{\substack{\ell \to \infty\\\ell \in K}} -c_1\| \nabla_x q_{\tau_k}(u^\ell,v^\ell) \|^2 \le 0,
	\end{equation*} 
	which implies that, for $\ell \in K$ sufficiently large, we have
	$
	\| \nabla_x q_{\tau_k}(u^\ell,v^\ell) \|\le \delta_k,
	$
	i.e., that the stopping criterion of step (S.1) is satisfied in a finite number of iterations,
	and this contradicts the fact that $\{( u^\ell,v^\ell ) \}$ is an infinite sequence. Condition \eqref{Eq:Innerk-1} is then satisfied by the stopping criterion, whereas condition \eqref{Eq:Innerk-2} follows by construction.
\end{proof}

In order for the theoretical analysis to hold, we only need to ensure that $H_\ell$ satisfies condition \eqref{bounded_eig}. This assumption can be guaranteed a priori by different ways of defining $H_\ell$. Among these valid choices, we can find classical setups leading back to iterations of standard algorithmic schemes such as gradient method ($H_\ell=I$), Newton method ($H_\ell=\nabla_{xx}^2 q_{\tau_k}(u^\ell,v^\ell)$, provided $f$ is uniformly convex), quasi-Newton methods and limited-memory BFGS type methods.

This aspect is crucial in practice: we are allowed to employ the most efficient solvers for nonlinear optimization to carry out step (S.3) of the Alternate Minimization algorithm and thus speed up the computation of step (S.2) of Algorithm \ref{Alg:PenDecomp}, which is the most burdensome one. As a comparison, the Augmented Lagrangian algorithm from \cite{JiaKanzowMehlitzWachsmuth2021} has to resort to a gradient-based method to solve the (constrained) sequential subproblems, possibly resulting in an inefficient method 
especially for ill-conditioned problems. Another difference
is pointed out in the following comment.

\begin{remark}\label{Rem:ModSubsolver}
The Augmented Lagrangian algorithm from \cite{JiaKanzowMehlitzWachsmuth2021} has to compute 
projections onto the set $ D $ within the computation of the
stepsizes, i.e., it may require many projections for a single
(inner) iteration. This is a notable difference to our
Algorithm~\ref{Alg:AM}, which requires only a single
projection after the computation of the new iterate $ u^{\ell + 1} $.
In fact, it would also be possible to apply several iterations
of an unconstrained optimization solver to the subproblem 
of minimizing the penalty function $ q_{\tau_k} ( \cdot, v^{\ell}) $
before updating the $ v $-component, i.e., before using a single
projection step.
\end{remark}

\section{Particular Instances}
\label{sec:realizations}

The idea of this section, similar to \cite{JiaKanzowMehlitzWachsmuth2021}, is to present some difficult 
optimization problems where projections onto the complicated
set $ D $ can be carried out easily. This section does not contain
any proofs since the corresponding results are known from the 
literature. However, since these particular instances will be
used in our numerical section, they have to be discussed 
in some detail.

\subsection{The case of Sparsity Constraints}
\label{sec:sparsity}
A particular case of problem \eqref{P} is that of sparsity constrained optimization problems, i.e., optimization problems of the form
\begin{equation}
\label{eq:sparse_prob}
\begin{aligned}
\min_{x\in\mathbb{R}^n}\;&f(x)\\\text{s.t. }& G(x)\in C,\\&x\in D = \{x\mid\|x\|_0\le s\},
\end{aligned}
\end{equation}
where $s<n$ and $\|x\|_0$ denotes the zero pseudo-norm of $x$, i.e., the number of nonzero components of $x$. The Penalty Decomposition approach was originally proposed in \cite{Lu2013} for this class of problems, and the inexact version was then proposed for the case $\{x\mid G(x)\in C\}=\mathbb{R}^n$ \cite{Lapucci2021}.

In fact, from the analysis in Section \ref{Sec:Alg}, we can deduce that the convergence results continue to hold for the inexact version of the algorithm even in presence of additional constraints.  

The Penalty Decomposition method is particularly appealing, from a computational perspective, for this class of problems since the Euclidean projection onto the sparse set $D$ is easily obtainable in closed form, as outlined e.g.\ in \cite{Lu2013,Lapucci2021}. Let us denote the index set of the largest $s$ variables at $\bar{x}$ in absolute value by $\mathcal{G}_s(\bar{x})$; for simplicity, we furthermore assume that cases of tie are handled unambiguously. Then, the projection of $\bar{x}$ onto $D$ is given by 
\begin{equation}
\label{eq:proj_sparse} (\Pi_D(\bar{x}))_i = \begin{cases}
\bar{x}_i &\text{if } i\in \mathcal{G}_s(\bar{x}),\\0&\text{otherwise.}
\end{cases}
\end{equation}
In other words, the projection can be simply computed by setting to zero the $n-s$ smallest components of $\bar{x}$

Note that M-stationarity, as defined in Definition~\ref{Def:MStat},  coincides with Lu-Zhang stationarity \cite{Lapucci2022}, which is the property guaranteed to hold for cluster points obtained by the original Penalty Decomposition method \cite{Lu2013}. Hence, we can conclude, from the results shown in Section~\ref{Sec:Alg}, that the inexact Penalty Decomposition method has the same convergence properties as its exact counterpart, and that the M-stationarity
concept includes a corresponding stationarity condition
particularly designed for cardinality constrained problems.
We note, however, that there exist further stationarity
concepts in this setting, see the corresponding discussions in, e.g., \cite{BeckEldar2013,Lu2013,Lapucci2022,Lammel2022}.

\subsection{Low-Rank Approximation Problems}
\label{sec:rank}

Here we consider the space $\mathbb{X}=\mathbb{R}^{m\times n}$ with given $n,m\in\mathbb{N}$, $n,m\ge 2$; equipped with the standard Frobenius inner product, this is a Euclidean space. 

In applications like computer vision, machine learning, computer algebra or signal processing, there is a strong interest in low-rank matrix optimization problems, see, e.g., \cite{Markovsky2012,Ben2001,Burer2002,Candes2009,Recht2010}. Specifically, letting $q=\min(m,n)$ and given $\kappa\le q-1$, we are interested in problems of the form
\begin{equation}
	\label{eq:prob_rank}
	\begin{aligned}
	\min_{X\in\mathbb{R}^{m\times n}}\;&f(X)\\\text{s.t. }& G(X)\in C,\\&X\in D = \{X\mid \text{rank}(X)\le \kappa\}.
	\end{aligned}
\end{equation}

The set $D$ has been thoroughly analyzed from a geometrical point of view, see e.g.\ \cite{Hosseini2019} for a formula for $\mathcal{N}_D^\text{lim}(X)$. Interestingly, elements of $\Pi_D(X)$ can be easily constructed exploiting the singular value decomposition of $X$ \cite{Markovsky2012,Zhang2011}.

\begin{proposition}
	Let $X\in\mathbb{X}=\mathbb{R}^{m\times n}$ and let $X=U\Sigma V^T$ its singular value decomposition, with orthogonal matrices $U\in\mathbb{R}^{m\times m}$, $V\in\mathbb{R}^{n\times n}$ and $\Sigma\in\mathbb{R}^{m\times n}$ diagonal with entries in non-increasing order, i.e.,
	$$\Sigma_{ij}=\begin{cases}
	\sigma_i&\text{if } i=j,\\0&\text{otherwise},
	\end{cases}\qquad \sigma_i\ge \sigma_j\;\forall \, i\ge j,$$
	being $\sigma_1,\ldots,\sigma_q$ the singular values of $X$.
	Moreover, let $\hat{\Sigma}$ the matrix obtained setting to zero the $q-\kappa$ bottom-right elements of $\Sigma$, i.e.,
	$$\hat{\Sigma}_{ij} = \begin{cases}
	\sigma_i&\text{if } i=j\le\kappa,\\0&\text{otherwise}.
	\end{cases}$$
	Then, $\hat{X} = U\hat{\Sigma}V^T\in \Pi_D(X)$.
\end{proposition}

Of course, the computation of the SVD for a matrix $X$ is not a costless operation, so obtaining an element of $\Pi_D(X)$, even though conceptually simple, requires a non-negligible amount of computing resources.

If we restrict the discussion to the case of symmetric positive semi-definite matrices, i.e., $D=\{X\in\mathbb{R}^{n\times n}\mid X\succeq 0,\; \text{rank}(X)\le \kappa\}$, we can resort to the eigenvalue decomposition instead of the SVD \cite{Zhang2011,JiaKanzowMehlitzWachsmuth2021}.

\begin{proposition}
	Let $X\in\mathbb{R}^{n\times n}$ be a symmetric matrix. Let us denote by $X=\sum_{i=1}^{n}\lambda_iv_iv_i^T$ its eigenvalue decomposition, where $\lambda_1\ge\ldots\ge\lambda_n$ are the non-increasingly ordered eigenvalues with corresponding eigenvectors $v_1,\ldots,v_n$. Then, we have $\hat{X} = \sum_{i=1}^{\kappa}\max\{0,\lambda_i\}v_iv_i^T\in\Pi_D(X)$.
\end{proposition}
We can thus observe that, in this particular case, in order to compute the projection onto the set $D$ we only need to find the $\kappa$ largest eigenvalues with the corresponding eigenvectors; this can be done efficiently, especially when $\kappa$ is small, as in most applications. 

A (exact) Penalty Decomposition scheme was developed in \cite{Zhang2011} to tackle low-rank optimization problems, exploiting the above closed form rules for projection onto $D$ both in the general and the positive semi-definite cases. The analysis in Section \ref{Sec:Alg} shows that the algorithmic framework maintains the same convergence properties even when the $X$-update step is carried out in an inexact fashion.

\subsection{Box-Switching Constrained Problems}

A wide class of relevant optimization problems with difficult geometric constraints is constituted by the so called \textit{box-switching} constrained problems \cite{JiaKanzowMehlitzWachsmuth2021} that can be formalized as follows:
\begin{equation}
\label{switching}
\begin{aligned}
\min_{x,y\in\mathbb{R}^n}\;&f(x,y)\\\text{s.t. }&G(x,y)\in C\\&(x,y)\in D=\{(x,y)\mid x_iy_i=0\;\forall\,i,\;l_x\le x\le u_x, \;l_y\le y\le u_y\},
\end{aligned}
\end{equation}
where, for simplicity, we assume that $l_x\le0 \le u_x$ and $l_y\le 0 \le u_y$.

This setting covers various disjunctive programming problems such as problems with
\begin{itemize}
	\item switching constraints \cite{Mehlitz2020b}: $l_x=l_y=-\infty$, $u_x=u_y=\infty$, \item complementarity constraints \cite{Ye1999}: $l_x=l_y=0$, $u_x=u_y=\infty$, \item relaxed sparsity constraints \cite{Burdakov2016}: $l_x=-\infty$, $u_x=\infty$, $l_y=0$, $l_y=1$.
\end{itemize} 
It is easy to realize that projection onto the set $D$ in this case is simple. Indeed, let us first consider the projection onto classical bound constraints $[l,u]$ of a vector $w$. Since the constraints are separable, we can immediately obtain the projection by computing, for each component $i$, the value
$$(P_{[l,u]}(w))_i=\begin{cases}
w_i&\text{if } l_i \le w_i\le u_i,\\
l_i&\text{if }w_i<l_i,\\ u_i&\text{if }w_i> u_i.
\end{cases}$$
With this in mind, noting that the set $D$ is also (pairwise) separable, we can obtain an element $(\hat{x},\hat{y})\in\Pi_D[(\bar{x},\bar{y})]$ by first computing
$$\tilde{x} = P_{[l_x,u_x]}(\bar{x}),\qquad \tilde{y} = P_{[l_y,u_y]}(\bar{y})$$ and then setting
$$(\hat{x}_i,\hat{y}_i) = \begin{cases}
(\tilde{x}_i,0)&\text{if } \bar{x}_i^2+(\tilde{y}_i-\bar{y}_i)^2\ge (\tilde{x}_i-\bar{x}_i)^2 + \bar{y}_i^2,\\ (0,\tilde{y}_i)&\text{otherwise}.
\end{cases}$$
Computing the projection onto $D$ thus amounts to computing $2n$ projections onto real intervals, which can be done with low computational effort. For this reason, a Penalty Decomposition type scheme again appears particularly appealing for this class of problems.

\subsection{General Disjunctive Programs}
\label{sec:disjoint}

A broad class of optimization problems with geometric constraints is represented by programs where variables are required to satisfy at least one among several sets of constraints:
\begin{equation}
	\label{eq:prob_disjoint}
	\begin{aligned}
		\min_{x\in\mathbb{R}^n}\;& f(x)\\
		\text{s.t. }& G(x)\in C,\\&x\in D = \bigcup_{i=1}^{N} D_i,
	\end{aligned}
\end{equation}
where $D_i$, $i=1,\ldots,N$ are closed convex sets. The resulting overall feasible set of these \textit{disjunctive programming problems} \cite{Flegel2007} typically takes the structure of a nonconvex, disconnected set.
Projections onto $D$ in this case can be computed by finding the closest among the $N$ projections onto $D_1,\ldots,D_N$:
$$\Pi_D(x) = \argmin_z\{\|z-x\|\mid z=P_{D_i}(x),\;i=1,\ldots,N\} .$$ 
Since, in general, the projection onto a convex set is already an expensive operation, the projection onto $D$ is consequently a costly task. We shall observe that, in fact, the settings analyzed in the previous subsections are particular instances of this setting where the peculiar structure of sets $D_i$ allows to efficiently compute the projection in smart ways.

The Penalty Decomposition approach might be appealing for problems of this form when the constraints $G(x)\in C$ are numerous and/or nontrivial and $N$ is also large. In these cases, the brute force strategy of solving $N$ problems with convex constraints may become computationally unsustainable and PD might represent an appealing alternative.

\section{Computational Experiments}
\label{sec:experiments}

In this section, we report the results of an extensive experimentation aimed at demonstrating the potential and the benefits of using the Penalty Decomposition algorithm on various classes of problems. The experiments have two main goals:
\begin{itemize}
	\item analyze the behavior of penalty decomposition in different settings and understand how to make it as efficient as possible; 
	\item compare the penalty decomposition approach with the augmented Lagrangian method proposed in \cite{JiaKanzowMehlitzWachsmuth2021}, which is, to the best of our knowledge, the only available algorithm from the literature designed to handle the general setting \eqref{P}. 
\end{itemize}

The code for the experiments has entirely been implemented in Python 3.9 and all the experiments have been run on a machine with the following specifications: Intel Xeon Processor E5-2430 v2, 6 physical cores (12 threads), 2.50 GHz, 16 GB RAM.

We considered benchmarks of problems from the classes discussed in Section \ref{sec:realizations}, i.e., cardinality constrained problems, low-rank approximation problems and disjunctive programming problems.

For the Penalty Decomposition approach (Algorithm \ref{Alg:PenDecomp}), we set an upper bound to the value of $\tau_k$ equal to $10^8$. We employed for the inner loop (Algorithm \ref{Alg:AM}) the stopping criterion
\begin{equation}
	\label{eq:stop:inner}
	q_{\tau_k}(u^\ell,v^\ell)-q_{\tau_k}(u^{\ell+1},v^{\ell+1})\le \epsilon_\text{in},
\end{equation}
whereas for the outer loop we employ
$$\|x^{k+1}-y^{k+1}\| + \text{dist}_C(G(x^{k+1})) \le \epsilon_\text{out}.$$
Both the above stopping conditions are the ones suggested in \cite{Lu2013} and we set $\epsilon_\text{in} = 10^{-5}$ and $\epsilon_\text{out} = 10^{-5}$. 

As unconstrained optimization solvers for the $x$-update step, we implemented the gradient descent algorithm with Armijo line search. We also ran experiments using the implementations of the conjugate gradient (CG, \cite{Bertsekas1997}), BFGS \cite{Bertsekas1997} and L-BFGS \cite{Liu1989} methods available in the \texttt{scipy} library.
For all of these algorithms, we used the stopping criterion $\|\nabla_x q_{\tau_k}(u^{\ell+1},v^\ell)\|\le \epsilon_\text{solv}$, with $\epsilon_\text{solv}=10^{-5}$ if not specified otherwise. 

As for the augmented Lagrangian method (ALM) from \cite{JiaKanzowMehlitzWachsmuth2021}, it employs as inner solver of subproblems the spectral gradient method (SGM) proposed in the same work. With reference to \cite[Algorithm 3.1]{JiaKanzowMehlitzWachsmuth2021}, we set $\sigma=10^{-5}$, $\gamma_0=1$, $\gamma_\text{max}=10^{12}$, $m=10$, $\tau=2$ (note that here it does not denote the penalty parameter). As for the ALM (\cite[Algorithm 4.1]{JiaKanzowMehlitzWachsmuth2021}), we set $\eta=0.8$. We employed the multipliers safeguarding technique, projecting the values obtained using the standard Hestenes-Powell-Rockafellar updates onto the box $[-10^{8}, 10^8]$.
For the spectral gradient loop, we used the stopping condition
$$\max_{j=0,\ldots,m-1} q_{\tau_k}(x^{\ell-m}) - \min_{j=0,\ldots,m}q_{\tau_k}(x^{\ell+1-m})\le 10^{-5},$$
where here we have used the notation of the present paper.
We used the same stopping condition \eqref{eq:stop:inner} as the PD method with $\epsilon_\text{in}=10^{-5}$ for the inner loop of the ALM, whereas for the outer loop we require $\text{dist}_C(G(x^{k+1}))\le \epsilon_\text{out}$, with $\epsilon_\text{out}=10^{-5}$. The stopping conditions have been chosen as similar as possible for the two algorithms, in order to have a fair comparison.

We also did experiments with a variant of our proposed approach, employing safeguarded Lagrange multipliers in an augmented Lagrangian fashion, i.e., we take the ALM approach from \cite{JiaKanzowMehlitzWachsmuth2021} and combine it with the decomposition idea to solve the resulting subproblems. The setting of multipliers and penalty parameter updates is the same as the one we employed for the ALM itself. In the following, we will show that this modification (denoted PDLM), which does not have any major impact in the convergence analysis, leads to significant benefits in practice. It is interesting to note that this finding is in contrast with the remarks that can be found in the conclusions of \cite{Lu2013}.

Finally, we point out that we did not report the values for the initial penalty parameter $\tau_0$ and its growing rate $\alpha_\tau$. Indeed, these parameter are quite crucial for the overall performance of both PD and ALM algorithms and have been suitably selected for each
class of problems. Thus, we will report each time the specific values of these two parameters.

\subsection{Sparsity Constrained Optimization Problems}

We begin our numerical analysis with sparsity constrained problems. The reason we start with this class of problems is twofold: a) the original PD approach was designed for these problems and b) results are more easily and intuitively interpretable.

We considered various experimental settings with problems of this class. Firstly, we begin with the simplest possible problems, i.e., convex quadratic problems with only sparsity constraints:
\begin{equation}
\label{eq:ccqp}
	\min_x\; \frac{1}{2}x^TQx+\nu c^Tx\quad\text{ s.t. } \|x\|_0\le s.
\end{equation}
We randomly generated instances of problem \eqref{eq:ccqp}, according to the following procedure:
\begin{gather}
\label{eq:gen_qp_a}
	Q = YDY,\quad Y=I-\frac{2}{\|y\|^2}yy^T,\quad y\in\mathbb{R}^n:y_i\sim \mathcal{U}(-1,1),\\\label{eq:gen_qp_b} D=\text{diag}(d_1,\ldots,d_n),\quad d_i = \exp\left(\frac{i-1}{n-1}n_{\text{cond}}\right),\quad c\in\mathbb{R}^n:c\sim\mathcal{U}(-1,1)
\end{gather}
where $n_\text{cond}$ denotes the desired condition number of the matrix $Q$.
We generated three instances with $n_\text{cond}=10$, $n\in\{10,25,50\}$, $s=3$, $\nu=5$ to evaluate the impact of different solvers for the $x$-update step on the alternating minimization scheme and, in turn, on the overall PD approach.

We report in Table \ref{tab:inner} the results obtained by running PD equipped with different inner solvers starting from the origin. We also ran the variant with Lagrange multipliers of our approach only with L-BFGS as inner solver; here we set $\tau_0=1$ and $\alpha_\tau=1.1$. Moreover, we considered  the spectral gradient method for comparison. Note that, since there are no additional constraints, there is no need to resort to the ALM.

\begin{table}[htb]
	\centering
	\caption{Results of experiments on three random instances of cardinality constrained quadratic problems \eqref{eq:ccqp}. Problems were generated according to \eqref{eq:gen_qp_a}-\eqref{eq:gen_qp_b} with $n_\text{cond}=10$, $s=3$, $\nu=5$.}
	\label{tab:inner}
	\begin{tabular}{|l|l|l|r|}\hline
		$\mathbf{n}$&\textbf{Algorithm}& \textbf{f\_val}&\textbf{runtime (s)}\\\hline
		\multirow{5}{*}{$10$}&\texttt{PD\_gd}&-9.70&3.68\\&\texttt{PD\_bfgs-scipy}&-9.70&0.53\\&\texttt{PD\_lbsfgs-scipy}&-9.70&0.31\\&\texttt{PDLM\_lbsfgs-scipy}&-9.63&0.26\\&\texttt{SGM}&-9.63&0.05\\\hline\multirow{5}{*}{$25$}&\texttt{PD\_gd}&-15.12&5.96\\&\texttt{PD\_bfgs-scipy}&-15.12&0.91\\&\texttt{PD\_lbsfgs-scipy}&-15.12&0.54\\&\texttt{PDLM\_lbsfgs-scipy}&-16.19&0.24\\&\texttt{SGM}&-15.58&0.01\\\hline
		\multirow{5}{*}{$50$}&\texttt{PD\_gd}&-23.92&9.17\\&\texttt{PD\_bfgs-scipy}&-23.92&1.69\\&\texttt{PD\_lbsfgs-scipy}&-23.92&0.89\\&\texttt{PDLM\_lbsfgs-scipy}&-23.92&0.31\\&\texttt{SGM}&-23.92&0.01\\\hline
	\end{tabular}
\end{table}

As expected, the use of quasi-Newton type solvers is highly beneficial:  BFGS leads to much faster convergence than the simple gradient method; the L-BFGS provides an additional, substantial speed up. The presence of Lagrange multipliers also seems to be beneficial, both in terms of efficiency and of quality of the obtained solution.
Based on these result, in the following we will always be using L-BFGS for the $x$-update step in Algorithm~\ref{Alg:AM}.

Note that the spectral gradient method clearly outperforms the PD approach in this case. This is indeed not surprising: being there no additional constraint, there is no need with the SGM to adopt a sequential penalty strategy, which is costly.

Next, we turn to a simple verification of the convergence properties of the PD approach. In particular, we consider the artificial example \cite[Example 2.2]{BeckEldar2013}, which is an instance of \eqref{eq:ccqp} with $Q= E+I$, being $E$ the matrix of all ones, $c = -(3,2,3,12,5)^T$, $\nu=1$ and $s=2$. We ran both PD and PDLM, with $\alpha_\tau=1.1$ from 1000 different starting points randomly generated in the hyper-box $[-10,10]^5$. We observed that the result strongly depends on the choice of $\tau_0$, as we report in Table \ref{tab:convergence}. Interestingly, both algorithms always converged to the global minimum $f(x^\star)=-41.33$ when we set $\tau_0=0.1$; in fact, we observed the same result for smaller values of $\tau_0$. On the other hand, as $\tau_0$ grows worse local minimizers become increasingly probable; the presence of Lagrange multipliers seems to alleviate, but not to suppress, this inconvenience. We argue that large values of $\tau_0$ make PD schemes more dependent on the starting point: since usually $x^0=y^0$, the penalty term is at the first iteration equal to $\tau_0\|x-x^0\|^2$, which binds variable $x$ close to the start.

\begin{table}[htbp]
	\centering
	\caption{Convergence of Penalty Decomposition methods on \cite[Example 2.2]{BeckEldar2013} for different values of $\tau_0$. We report the number of times an objective value has been obtained out of 1000 runs from different starting points chosen randomly in $[-10,10]^n$.} 
	\label{tab:convergence}
	\begin{tabular}{|l|l|l|l|l|l|}
		\hline
		$\boldsymbol{\tau_0}$&\textbf{Algorithm}&$\mathbf{f=-41.33}$&$\mathbf{f=-39}$&$\mathbf{f=-36.33}$&\textbf{Others}\\\hline
		\multirow{2}{*}{$0.1$}&PD&1000&0&0&0\\&PDLM&1000&0&0&0 \\\hline
		\multirow{2}{*}{$1$}&PD&661&339&0&0\\&PDLM&1000&0&0&0 \\\hline
		\multirow{2}{*}{$10$}&PD&683&275&42&0\\&PDLM&620&326&54&0 \\\hline
		\multirow{2}{*}{$100$}&PD&549&250&40&161\\&PDLM&527&295&57&121\\\hline
	\end{tabular}
\end{table}

At this point, we have devised a setting that apparently makes the PD approach efficient and effective. We therefore expect the algorithm to indeed be a good choice to resort to when: a) additional constraints are present and/or b) the projection operator is costly. In the former case, SGM needs to be employed within another sequential scheme, namely, the ALM, which is the only alternative to the PD available from the literature; in the latter case, the advantage of PD over the ALM may not be straightforward. In fact, the two algorithms share a similar structure, sequentially solving penalized subproblems; in order to do so, unconstrained continuous optimization steps and projections onto $D$ are repeatedly carried out; however, in PD many descent steps can be carried out before turning to the projection step; on the contrary, in the ALM we need to do the projection after every gradient step (in fact, we do it many times per iteration of the SGM to satisfy the acceptance criterion), cf.\ the discussion
in Remark~\ref{Rem:ModSubsolver}.

We now turn to sparsity constrained problems with additional constraints. In particular, we keep considering convex quadratic problems, but with simplex constraints, i.e., problems of the form
\begin{equation*}
\label{eq:portfolio}
	\min_x\;\frac{1}{2}x^TQx + \nu c^Tx,\quad\text{ s.t. } e^Tx = 1,\quad x\ge0,\quad \|x\|_0\le s,
\end{equation*}
where $e\in\mathbb{R}^n$ denotes the vector of all ones. This is a classical sparse portfolio optimization problem \cite{Bertsimas2022}, where $Q$ and $c$ denote the covariance matrix and the mean of $n$ possible assets. 

Portfolio optimization problems are particularly useful to test the proposed algorithm since we can easily obtain the global optimum to be used as a reference. Indeed, we can do so exploiting the mixed-integer reformulation of the problem with binary indicator variables and big-M type constraints and employing efficient software solvers such as Gurobi \cite{Gurobi}.

We first consider synthetic problems. Using \eqref{eq:gen_qp_a}-\eqref{eq:gen_qp_b}, we generated 10 problems for each combination of $n\in\{20,40,60\}$ and $n_\text{cond}\in\{10,100, 500\}$, for a total of 90 problems. We set $s=4$ when $n=20$, $s=7$ for $n=40$ and $s=9$ for $n=60$. We set $\nu=1$ and use as starting point of the experiments $\tilde{x} = (1/n,\ldots,1/n)^T$; we ran PD, PDLM, ALM all with $\tau_0=1$ and $\alpha_\tau=1.1$. We also ran Gurobi on all instances to obtain the global optimizer to be used as reference; note that Gurobi indeed finds the certified global minimum in tens of seconds.
The overall results of the experiments are reported in Figure \ref{fig:pp_ccqp}. The results concerning efficiency (runtime) are presented in the form of performance profiles \cite{Dolan2002} in Figure \ref{fig:pp_ccqp_time}. We can observe that Penalty Decomposition with Lagrange multipliers is generally faster than the other two considered algorithms.
As for the quality of the retrieved solutions, we plot in Figure \ref{fig:pp_ccqp_fval} the cumulative distribution of the relative gap between the solution found by a solver and the certified global optimum found with Gurobi; the result of PDLM is surprisingly remarkable, as it almost always reached a value very close, and often equal to, the global optimum; on the contrary, both PD and the ALM end up with substantially suboptimal solutions in almost a half of the cases.

\begin{figure}[htbp]
	\centering
	\begin{subfigure}[t]{0.49\textwidth}
		\centering
		\includegraphics[width=\textwidth]{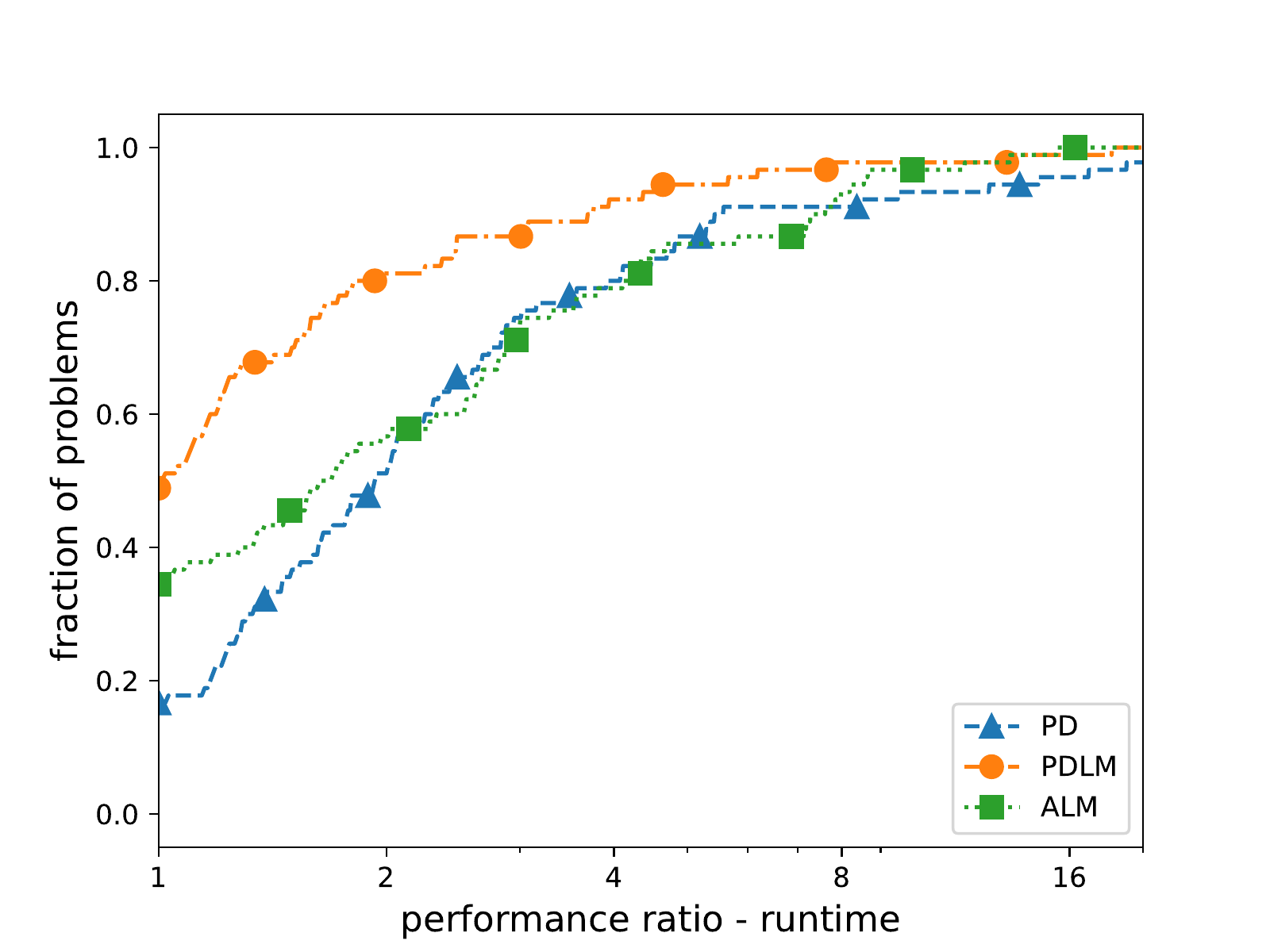}
		\caption{Perfromance profile of runtimes.}
		\label{fig:pp_ccqp_time}
	\end{subfigure}
	\hfill
	\begin{subfigure}[t]{0.49\textwidth}
		\centering
		\includegraphics[width=\textwidth]{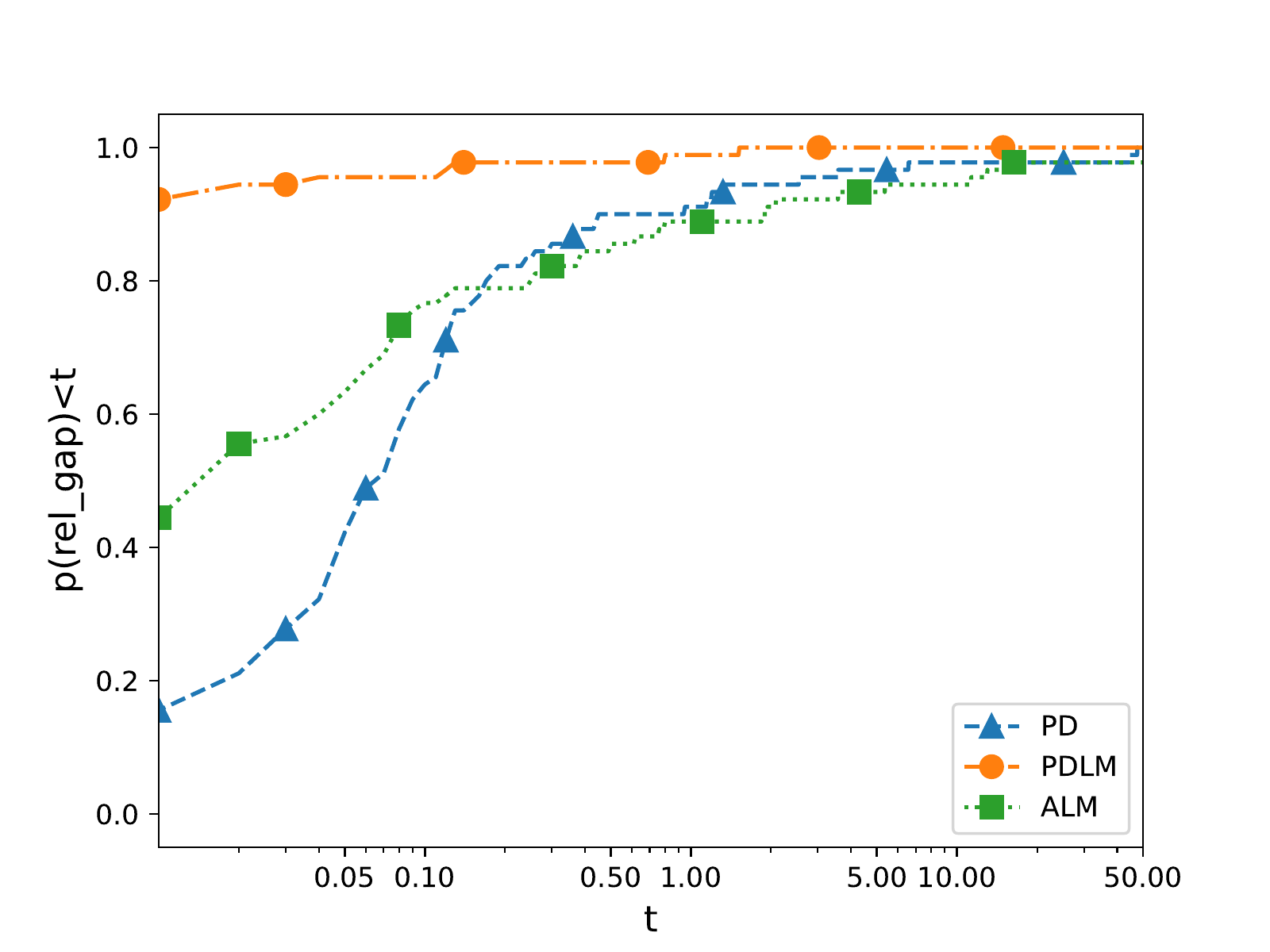}
		\caption{Cumulative distribution of relative gap to the certified global optimum.}
		\label{fig:pp_ccqp_fval}
	\end{subfigure}
	\caption{Results of experiments on 90 randomly generated sparse portfolio selection problems, running PD, PDLM and the ALM.}
	\label{fig:pp_ccqp}
\end{figure}

We conclude the analysis on sparsity constrained problems looking at the results on 6 instances of real world portfolio selection problems. In particular, the data used in the experiments consists of daily data for securities from the FTSE 100
index, from 01/2003 to 12/2007. The three datasets are referred to as DTS1, DTS2, and
DTS3, and are formed by $n=12$, 24, and 48 securities, respectively. We also included three
datasets from the Fama/French benchmark collection (FF10, FF17, and FF48, with $n$ equal to 10, 17, and 48), using the monthly returns from 07/1971 to 06/2011. The datasets are generated as in \cite{Cocchi2020}. For each dataset, we define an instance of problem \eqref{eq:portfolio}: the values of $s$ and $\nu$ are set as reported in Table \ref{tab:portfolio_real}, and are such that the cardinality constraint is active at the optimal solution. We used again $\tilde{x} = (1/n,\ldots,1/n)^T$ as starting point.
As for the penalty parameter, we set $\tau_0=0.01$ and for the Penalty Decomposition methods, whereas we found that a larger value $\tau_0=1$ was beneficial for the ALM. The parameter $\alpha_\tau$ was set to $1.01$ for all methods.
The results are reported in Table \ref{tab:portfolio_real} and we can observe that the trends outlined by the previous experiments are substantially confirmed.

\begin{table}[htb]
	\centering
	\caption{Results of experiments on 6 real world sparse portfolio selection problems, solved using different algorithmic approaches.}
	\label{tab:portfolio_real}
	\begin{tabular}{|c|l|r|r|}
		\hline
		\textbf{Problem ($\mathbf{s,\boldsymbol{\nu}}$)}&\textbf{Algorithm}&\textbf{f\_val}&\textbf{runtime (s)}\\\hline
		\multirow{4}{*}{DTS1 $(2,0.001)$}&Gurobi&4.10e-05&0.02\\&PD&4.27e-05&2.01\\&PDLM&4.25e-05&2.07\\&ALM&4.69e-05&8.72
		\\\hline
		\multirow{4}{*}{DTS2 $(4,0.001)$}&Gurobi&2.52e-05&0.04\\&PD&2.75e-05&2.49\\&PDLM&2.75e-05&1.55\\&ALM&2.75e-05&5.36
		\\\hline
		\multirow{4}{*}{DTS3 $(6,0.001)$}&Gurobi&2.19e-05&0.19\\&PD&2.43e-05&4.30\\&PDLM&2.42e-05&1.72\\&ALM&2.48e-05&5.98
		\\\hline
		\multirow{4}{*}{FF10 $(2,0.05)$}&Gurobi&2.87e-05&0.01\\&PD&3.07e-05&2.43\\&PDLM&2.96e-05&3.38\\&ALM&2.87e-05&7.21
		\\\hline
		\multirow{4}{*}{FF17 $(2,0.05)$}&Gurobi&2.08e-05&0.01\\&PD&3.34e-05&3.26\\&PDLM&3.11e-05&3.98\\&ALM&3.46e-05&47.89
		\\\hline
		\multirow{4}{*}{FF48 $(5,0.05)$}&Gurobi&-1.10e-05&0.06\\&PD&1.59e-05&11.09\\&PDLM&-9.30e-06&5.49\\&ALM&9.66e-05&0.47\\\hline
	\end{tabular}
\end{table}

\subsection{Low-Rank Optimization Problems}

In this section, we study problems as discussed in Section \ref{sec:rank} where $\mathbb{X}=\mathbb{R}^{m\times n}$ and $D$ consists of a low-rank matrices space.

To begin with, we consider the class of \textit{nearest low-rank correlation matrix problems}, which was already used as a benchmark in \cite{Zhang2011}. In detail, the problem can be formulated as
\begin{equation*}
	\min_{X\in\mathbb{R}^{n\times n}}\frac{1}{2}\|X-A\|_F^2\quad \text{s.t.}\quad X^T=X, \; X\succeq 0,\; \text{diag}(X) = e,\; \text{rank}(X)\le \kappa,
\end{equation*} 
where $A$ is a given symmetric correlation matrix. The test problems we consider are the same as in \cite{Zhang2011}, and their corresponding matrix $A$ is defined as follows:
\begin{itemize}
	\item (P1) $A_{ij} = 0.5+0.5\exp(-0.05|i-j|)$ for all $i,j$;
	\item (P2) $A_{ij} = \exp(-|i-j|)$ for all $i,j$;
	\item (P3) $A_{ij} = 0.6+0.4\exp(-0.1|i-j|)$ for all $i,j$.
\end{itemize}   
For each of the above problems, we considered the instances with $n=200$ and $n=500$ and a value of $\kappa=5,10$ and $20$.

We experimentally compared the ALM and some implementations of the Penalty Decomposition approach; for all these algorithms, we set $\tau_0=1$ and $\alpha_\tau= 1.2$. We also needed for these experiments to set the upper bound on the value of $\tau_k$ to $10^{12}$.

Note that solving the $X$-update subproblem with an iterative solver has a significant cost, as we are considering problems with up to $n\times n =250000$ variables; moreover, we are dealing with ill-conditioned quadratic problems, thus we found convenient switching from L-BFGS to the CG method.
We tested two settings for the $X$-update step with CG: a strongly inexact setting, where the CG method is stopped after at most 5 steps (\texttt{PD-cg-inaccurate}), or when the norm of the gradient is smaller than $\epsilon_\text{solv}=0.1$, and a more accurate setting, where up to 20 CG steps are carried out and the tolerance for the gradient norm stopping condition is set to $0.001$ (\texttt{PD-cg-accurate}).

In addition, we note that, in fact, the $X$-update subproblem
$$\min_{\substack{X\in\mathbb{R}^{n\times n}\\X=X^T}} \frac{1}{2}\|X-A\|_F^2 + \frac{\tau}{2}(\|X-Y\|_F^2+\|\text{diag}(X)-e\|^2)$$
can be solved to global optimality in closed form; we thus also carried out experiments with this option (\texttt{PD-exact}); moreover, we also consider the strategy adopted in \cite{Zhang2011}, where the constraint $\text{diag}(X)=e$ is kept as a lower-level constraint and the $X$-update subproblem is still solved in closed form (\texttt{PD-exact-lower-level}).

We finally report that we found the introduction of Lagrange multipliers associated with constraints $G(x)\in C$ useful. We instead noticed that multipliers associated with the constraint $X=Y$ are not helpful. This observation is in line with the work in \cite{Zhang2011}, where only the constraint $X=Y$ was in practice handled by the penalty approach and multipliers were reported not to be beneficial. 
In the experiments described in the following, only multipliers associated with the original problem constraints have been employed.
The results of the experiment are reported in Tables \ref{tab:correlation_p1}, \ref{tab:correlation_p2} and \ref{tab:correlation_p3}.

\begin{table}[htbp]
	\centering
	\caption{Results of experiments on nearest low-rank correlation matrix problem (P1); for each instance we report the runtime, objective value and number of (inner) iterations for each considered solver.}
	\label{tab:correlation_p1}
	\begin{tabular}{|c|l|r|r|r|}
		\hline
		\textbf{Problem ($\mathbf{n,\boldsymbol{\kappa}}$)}&\textbf{Algorithm}&\textbf{f\_val}&\textbf{runtime (s)}&\textbf{iter}
		\\\hline
		\multirow{5}{*}{P1$(200,5)$}
		&\texttt{PD-cg-inaccurate}&183.8&19.6&1539\\
		&\texttt{PD-cg-accurate}&183.7&104.3&8470\\
		&\texttt{PD-exact}&183.7&37.8&9179\\
		&\texttt{PD-exact-lower-level}&183.7&20.2&6494
		\\&\texttt{ALM}&183.7&124.8&23583
		\\\hline
		\multirow{5}{*}{P1$(200,10)$}
		&\texttt{PD-cg-inaccurate}&27.7&9.5&597\\
		&\texttt{PD-cg-accurate}&27.6&22.11&3277\\
		&\texttt{PD-exact}&27.6&22.7&3587\\
		&\texttt{PD-exact-lower-level}&27.6&12.4&2606
		\\&\texttt{ALM}&27.6&25.8&3705
		\\\hline
		\multirow{5}{*}{P1$(200,20)$}
		&\texttt{PD-cg-inaccurate}&3.7&8.7&273\\
		&\texttt{PD-cg-accurate}&3.5&37.1&1194\\
		&\texttt{PD-exact}&3.5&13.7&1332\\
		&\texttt{PD-exact-lower-level}&3.5&7.5&1006
		\\&\texttt{ALM}&3.5&31.1&2261
		\\\hline
		\multirow{5}{*}{P1$(500,5)$}
		&\texttt{PD-cg-inaccurate}&3108.0&466.7&5602\\
		&\texttt{PD-cg-accurate}&3107.0&2396.3&29689\\
		&\texttt{PD-exact}&3107.0&619.4&33654\\
		&\texttt{PD-exact-lower-level}&3107.0&339.4&23053
		\\&\texttt{ALM}&3107.0&1710.1&47539
		\\\hline
		\multirow{5}{*}{P1$(500,10)$}
		&\texttt{PD-cg-inaccurate}&748.2&543.5&2132\\
		&\texttt{PD-cg-accurate}&748.2&3410.9&12673\\
		&\texttt{PD-exact}&748.2&356.2&14299\\
		&\texttt{PD-exact-lower-level}&748.2&207.6&9846
		\\&\texttt{ALM}&748.2&1351.5&17322		
		\\\hline
		\multirow{5}{*}{P1$(500,20)$}
		&\texttt{PD-cg-inaccurate}&123.7&200.2&811\\
		&\texttt{PD-cg-accurate}&123.4&1425.6&5078\\
		&\texttt{PD-exact}&123.4&216.0&5787\\
		&\texttt{PD-exact-lower-level}&123.4&127.5&4077
		\\&\texttt{ALM}&123.4&1201.9&13568
		\\\hline
	\end{tabular}
\end{table}

\begin{table}[htbp]
	\centering
	\caption{Results of experiments on nearest low-rank correlation matrix problem (P2); for each instance we report the runtime, objective value and number of (inner) iterations for each considered solver.}
	\label{tab:correlation_p2}
	\begin{tabular}{|c|l|r|r|r|}
		\hline
		\textbf{Problem ($\mathbf{n,\boldsymbol{\kappa}}$)}&\textbf{Algorithm}&\textbf{f\_val}&\textbf{runtime (s)}&\textbf{iter}
		\\\hline
		\multirow{5}{*}{P2$(200,5)$}
		&\texttt{PD-cg-inaccurate}&3701.2&51.3&4312\\
		&\texttt{PD-cg-accurate}&3700.8&181.7&15415\\
		&\texttt{PD-exact}&3700.8&63.4&16360\\
		&\texttt{PD-exact-lower-level}&3700.8&32.6&10724
		\\&\texttt{ALM}&3700.8&156.4&22519
		\\\hline
		\multirow{5}{*}{P2$(200,10)$}
		&\texttt{PD-cg-inaccurate}&1703.7&29.8&2162\\
		&\texttt{PD-cg-accurate}&1703.1&100.85&7678\\
		&\texttt{PD-exact}&1703.1&45.0&8171\\
		&\texttt{PD-exact-lower-level}&1703.1&24.1&5384
		\\&\texttt{ALM}&1703.1&136.7&14017
		\\\hline
		\multirow{5}{*}{P2$(200,20)$}
		&\texttt{PD-cg-inaccurate}&712.2&30.9&1065\\
		&\texttt{PD-cg-accurate}&712.0&118.2&3738\\
		&\texttt{PD-exact}&712.0&35.0&3969\\
		&\texttt{PD-exact-lower-level}&712.0&19.2&2629
		\\&\texttt{ALM}&712.0&117.1&6519
		\\\hline
		\multirow{5}{*}{P2$(500,5)$}
		&\texttt{PD-cg-inaccurate}&24249.5&950.8&11449\\
		&\texttt{PD-cg-accurate}&24248.2&3463.7&42892\\
		&\texttt{PD-exact}&24248.2&836.4&46711\\
		&\texttt{PD-exact-lower-level}&24248.2&447.9&30240
		\\&\texttt{ALM}&24248.2&627.6&15633
		\\\hline
		\multirow{5}{*}{P2$(500,10)$}
		&\texttt{PD-cg-inaccurate}&11752.9&501.5&5754\\
		&\texttt{PD-cg-accurate}&11749.1&1845&21508\\
		&\texttt{PD-exact}&11749.1&560.9&23576\\
		&\texttt{PD-exact-lower-level}&11749.1&308.4&15317
		\\&\texttt{ALM}&11749.1&2289.7&44857		
		\\\hline
		\multirow{5}{*}{P2$(500,20)$}
		&\texttt{PD-cg-inaccurate}&5505.0&283.1&2878\\
		&\texttt{PD-cg-accurate}&5502.9&1049.6&12854\\
		&\texttt{PD-exact}&5502.9&420.1&11932\\
		&\texttt{PD-exact-lower-level}&5502.9&238.6&7834
		\\&\texttt{ALM}&5502.9&2642.8&360565
		\\\hline
	\end{tabular}
\end{table}

\begin{table}[htbp]
	\centering
	\caption{Results of experiments on nearest low-rank correlation matrix problem (P3); for each instance we report the runtime, objective value and number of (inner) iterations for each considered solver.}
	\label{tab:correlation_p3}
	\begin{tabular}{|c|l|r|r|r|}
		\hline
		\textbf{Problem ($\mathbf{n,\boldsymbol{\kappa}}$)}&\textbf{Algorithm}&\textbf{f\_val}&\textbf{runtime (s)}&\textbf{iter}
		\\\hline
		\multirow{5}{*}{P3$(200,5)$}
		&\texttt{PD-cg-inaccurate}&265.1&22.2&1746\\
		&\texttt{PD-cg-accurate}&265.0&108.7&9224\\
		&\texttt{PD-exact}&265.0&40.3&9934\\
		&\texttt{PD-exact-lower-level}&265.0&21.5&6937
		\\&\texttt{ALM}&265.0&131.6&21078
		\\\hline
		\multirow{5}{*}{P3$(200,10)$}
		&\texttt{PD-cg-inaccurate}&56.1&24.2&704\\
		&\texttt{PD-cg-accurate}&56.1&103.7&3277\\
		&\texttt{PD-exact}&56.1&24.8&3587\\
		&\texttt{PD-exact-lower-level}&56.1&13.5&2606
		\\&\texttt{ALM}&56.1&82.6&3705
		\\\hline
		\multirow{5}{*}{P3$(200,20)$}
		&\texttt{PD-cg-inaccurate}&9.1&9.29&305\\
		&\texttt{PD-cg-accurate}&8.5&45.4&1474\\
		&\texttt{PD-exact}&8.5&16.1&1625\\
		&\texttt{PD-exact-lower-level}&8.5&8.9&1196
		\\&\texttt{ALM}&8.5&109.9&7875
		\\\hline
		\multirow{5}{*}{P3$(500,5)$}
		&\texttt{PD-cg-inaccurate}&2871.4&1451.3&5607\\
		&\texttt{PD-cg-accurate}&2869.3&7902.8&29897\\
		&\texttt{PD-exact}&2869.3&615.6&34110\\
		&\texttt{PD-exact-lower-level}&2869.3&360.8&23229
		\\&\texttt{ALM}&2869.3&300.57&4128
		\\\hline
		\multirow{5}{*}{P3$(500,10)$}
		&\texttt{PD-cg-inaccurate}&982.1&219.9&2460\\
		&\texttt{PD-cg-accurate}&981.8&1182.7&13657\\
		&\texttt{PD-exact}&981.8&371.1&16322\\
		&\texttt{PD-exact-lower-level}&981.8&209.9&10463
		\\&\texttt{ALM}&981.8&444.7&9343		
		\\\hline
		\multirow{5}{*}{P3$(500,20)$}
		&\texttt{PD-cg-inaccurate}&243.8&102.2&964\\
		&\texttt{PD-cg-accurate}&243.7&1049.6&10757\\
		&\texttt{PD-exact}&243.7&420.1&11932\\
		&\texttt{PD-exact-lower-level}&243.7&238.6&7834
		\\&\texttt{ALM}&243.7&2642.8&36056
		\\\hline
	\end{tabular}
\end{table}

We can observe that the exact versions of the PD approach are the best performing ones from all perspectives, with the \texttt{PD-exact-lower-level} originally used in \cite{Zhang2011} standing out. This is in fact not surprising: in this case the exact method solves subproblems not only with higher accuracy, but also employing much less time than using an iterative solver.

Interestingly, however, we observe that the ``inaccurate'' version of the inexact PD attains runtimes that are comparable with the exact approaches, with only small drops in the quality of the retrieved solution. On the other hand, with a slightly more accurate inexact minimization we are always able to retrieve the best solution as the exact methods, with a computational effort generally comparable to that of the ALM.

We can thus deduce that a suitable configuration exists for the inexact PD approach that provides a good trade-off between solution quality and runtime.

These results are encouraging for all those settings where the exact version of the Penalty Decomposition approach is not employable by construction. 

\bigskip

We then turn to a new class of problems, where matrices are not symmetric positive semi-definite and the $X$-update step requires a solver to be carried out. Specifically, we consider the \textit{low-rank based multi-task training \cite{Zhang2021} of logistic models \cite{Hastie2009}}. Given a collection of somewhat related binary classification tasks $\mathcal{T}_1,\ldots,\mathcal{T}_m$, $\mathcal{T}_i=\{(X_i,Y_i)\mid X_i\in\mathbb{R}^{N_i\times n},\;Y_i\in\{0,1\}^{N_i}\}$, where $X_i$ represents the data matrix for each task and $Y_i$ the corresponding labels, we can formalize the multitask logistic regression training problem as   
\begin{equation}
	\label{eq:mtlr}
	\min_{W,U,V\in\mathbb{R}^{m\times n}} \sum_{t=1}^{m}\mathcal{L}(W_t;X_t,Y_t)+\eta\|U\|_F^2\quad \text{s.t. } \text{rank}(V)\le\kappa,\quad W=U+V,
\end{equation}
where the $t$-th row $W_t$ of $W$ denotes the weights of the logistic model for the $t$-th task; each model is defined as the sum of a component independently characterizing the particular task, which is regularized, and a second component that lies in a linear subspace shared by all tasks. The stronger is the regularization parameter $\eta$, the higher will be the similarity of the obtained models. By $\mathcal{L}(W_t;X_t,Y_t)$ we denote the binary cross entropy loss function of the logistic model for task $t$, which is a convex function that, however, cannot be minimized in closed form.

We can easily observe that the problem can be solved by Penalty Decomposition, duplicating variable $V$. The constraint $W=U+V$ can also be tackled by the penalty approach. The update step of the original variables $W,U,V$ cannot be carried out in closed form, thus we need to resort to the inexact version of the PD method.

For the experiments, we used the \texttt{landmine} dataset \cite{Xue2007}, consisting of 9-dimensional data points representing radar images from 29 landmine fields/tasks. Each task aims to classify points as landmine or clutter. There are 14,820 data points in total. Tasks can approximately be clustered into two classes of ground surface conditions, so we expect $r=2$ to be a reasonable bound for the low-rank component of the solution.
We defined four instances of problem \eqref{eq:mtlr}, corresponding to values of $\eta$ in $\{0.01,0.1, 0.5,2\}$.
We examined the behavior of the inexact PD and ALM methods under different parameters configurations. In particular, we considered the following settings:
\begin{itemize}
	\item Penalty Decomposition
	\begin{itemize}
		\item Lagrange multipliers associated with all constraints;
		\item $\tau_0=10^{-3}$, $\alpha_\tau=1.3$;
		\item conjugate gradient (CG) for $x$-update steps;
		\item three options for CG termination criteria:
		\begin{itemize}
			\item $\epsilon_\text{solv}=0.1$, $\text{max\_iters}_{\text{CG}}=5$ (\texttt{pd\_inaccurate});
			\item $\epsilon_\text{solv}=0.05$, $\text{max\_iters}_{\text{CG}}=8$ (\texttt{pd\_mid});
			\item $\epsilon_\text{solv}=0.001$, $\text{max\_iters}_{\text{CG}}=20$ (\texttt{pd\_accurate});
		\end{itemize}
	\end{itemize} 
	\item ALM
	\begin{itemize}
		\item $\tau_0=1$, $\alpha_\tau=1.3$;
		\item two options for spectral gradient parameters:
		\begin{itemize}
			\item $\epsilon_\text{in} = 10^{-1}$, $m=1$, $\gamma_\text{max}=10^6$, $\sigma=0.05$ (\texttt{alm\_fast});
			\item $\epsilon_\text{in} = 10^{-3}$, $m=4$, $\gamma_\text{max}=10^9$, $\sigma=5\cdot 10^{-4}$ (\texttt{alm\_accurate}).
		\end{itemize}
	\end{itemize}
\end{itemize}  
We also report the results obtained by optimizing each task independently. For all PD and ALM configurations, we used as starting solution the one retrieved by single task optimization. Note that both configurations for ALM have lower precision than the default one reported at the beginning of Section \ref{sec:experiments}; with this particular problem, we found the default configuration to be remarkably inefficient; however we shall underline that, in other test cases considered in this paper, these alternative configurations had led to convergence issues concerning numerical errors. The results of the experiment are reported in Figure \ref{fig:landmine}. Note that here we are interested in the optimization process metrics, not in the out-of-sample prediction performance of the obtained models. 
\begin{figure}[ht]
	\centering
	\includegraphics[width=0.95\textwidth]{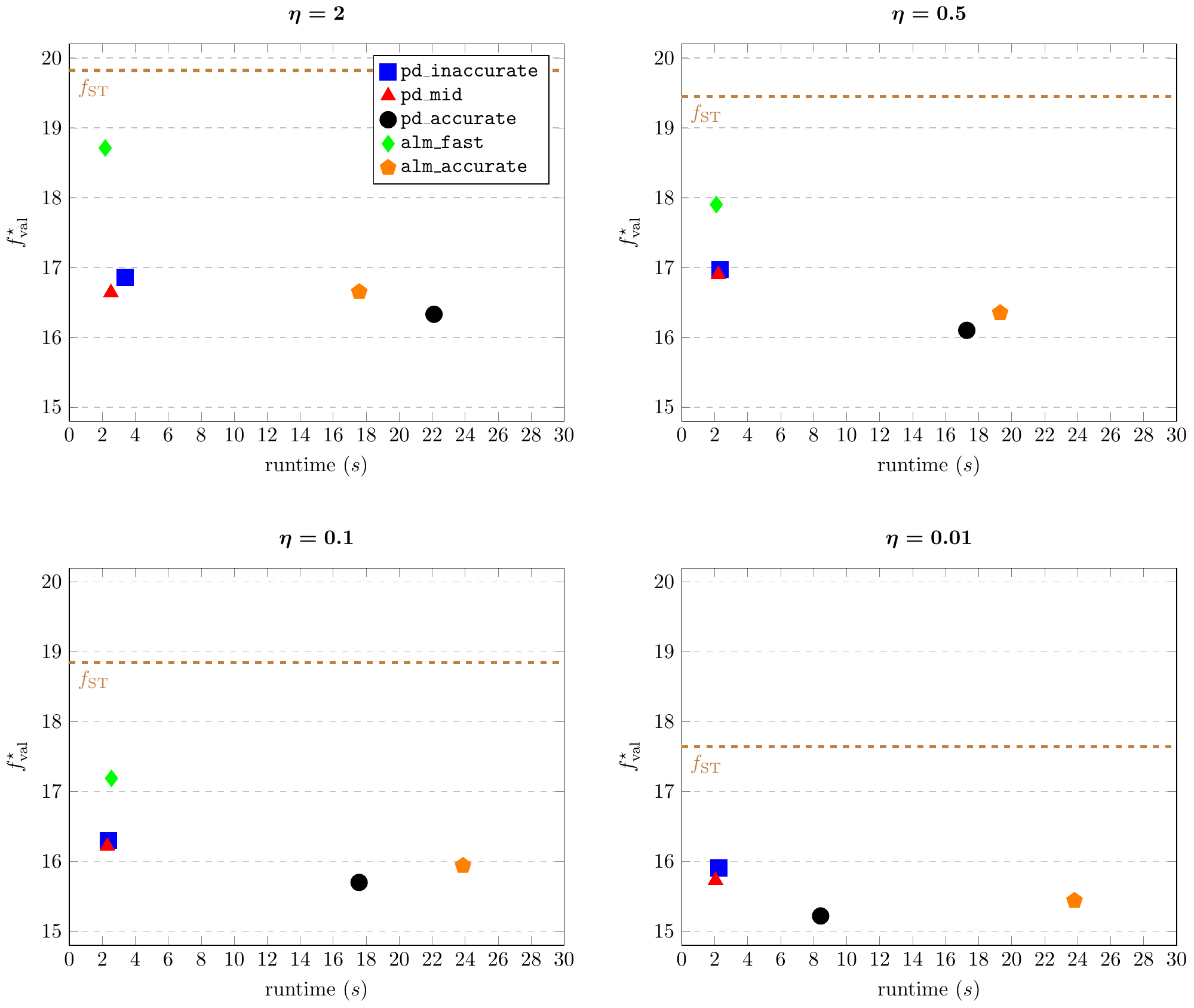}
	\caption{Runtime/quality trade-off for different set-ups of PD and ALM on low-rank multitask logistic regression problems. The four problems are obtained from the \texttt{landmine} dataset for different values of the regularization parameter $\eta$ and setting $\kappa=2$.}
	\label{fig:landmine}
\end{figure}

We can observe that different setups for the algorithms allow to obtain different trade-offs between speed and solution quality. In particular, for the PD method we see that the trend observed with the correlation matrix problems are confirmed: solving the $x$-update subproblem up to lower accuracy allows to save computing time but at the cost of small yet not negligible sacrifice on the solution quality. A similar and even stronger trend can be observed for the ALM. Finally, we can observe that PD appears to be superior to the ALM both in terms of efficiency and effectiveness.

\subsection{Disjunctive Programming Problems}

In this section, we computationally analyze the performance of the Penalty Decomposition approach on problems of the form \eqref{eq:prob_disjoint}. The main goal of this section is to compare the performance of inexact PD with that of the ALM algorithm in a setting where the projection operation is costly and is in fact responsible for the largest part of the computational burden: as highlighted in Section \ref{sec:disjoint}, it amounts to compute the projection onto each of the convex sets $D_i$.

For the experiments, we defined the following test problem:
\begin{align*}
	\min_{x\in\mathbb{R}^n}\;& \mathcal{L}(x)\\ \text{s.t. }&x\in\bigcup_{q=1}^{N}\{x\mid A_qx\le b_q\}\\&\sum_{i=1}^{n}c_{ij}(x_i-p_{ij})^4\le t_j,\quad j=1,\ldots,m, 
\end{align*}
where $\mathcal{L}$ denotes the (convex) logistic loss on a randomly generated dataset of $200$ examples. We assume $A_q\in\mathbb{R}^{s\times n}$ and $b_q\in \mathbb{R}^s$ have the same dimensions for $q=1,\ldots,N$ and their coefficients are uniformly drawn from $[-1,1]$; the coefficients $c$ are randomly picked from $[0,1]$, whereas values for $p$ are from $[-0.5,0.5]$. We set $t_j=0.1$ for all $j$.

First, we consider the problem with $n=10$, $s=12$ and $m=1$, for values of $N$ varying in $\{2,5,10,20,50,100\}$.  In Table \ref{tab:disjoint} we report the results obtained running PD (parameters: $\tau_0=0.1$, $\alpha_\tau=1.2$, $\epsilon_\text{in}=0.01$, Lagrange multipliers employed) and the ALM (parameters $\tau_0=1$, $\alpha_\tau=1.2$ $m=4$, $\sigma=0.01$, $\epsilon_\text{in}=0.01$), together with reference values obtained with the enumeration approach (subproblems are solved using the SLSQP method available in \texttt{scipy}), which allows to retrieve the certified global optimizer. For the projection steps onto sets $D_i$ we used \texttt{gurobi} solver.

\begin{table}[htb]
	\centering
	\caption{Results of experiments on disjoint programming problems, for increasing number of feasible set components $N$. }
	\label{tab:disjoint}
	\begin{tabular}{|c|l|r|r|}
		\hline
		\textbf{$\boldsymbol{N}$ }&\textbf{Algorithm}&\textbf{f\_val}&\textbf{runtime (s)}\\\hline
		\multirow{3}{*}{2}&Enumeration + SLSQP&140.58&0.61\\&PD&140.58&2.10\\&ALM&140.64&12.02
		\\\hline
		\multirow{3}{*}{5}&Enumeration + SLSQP&139.40&2.37\\&PD&139.40&6.92\\&ALM&139.59&67.22
		\\\hline
		\multirow{3}{*}{10}&Enumeration + SLSQP&135.82&5.87\\&PD&135.82&5.87\\&ALM&135.83&48.13
		\\\hline
		\multirow{3}{*}{20}&Enumeration + SLSQP&134.94&9.96\\&PD&134.94&19.41\\&ALM&134.95&170.31
		\\\hline
		\multirow{3}{*}{50}&Enumeration + SLSQP&135.08&18.44\\&PD&135.08&41.66\\&ALM&135.08&267.95
		\\\hline
		\multirow{3}{*}{100}&Enumeration + SLSQP&135.69&30.03\\&PD&135.69&61.26\\&ALM&135.69&465.91
		\\\hline
	\end{tabular}
\end{table}

We can observe that PD was always much faster than the ALM; moreover, it always ended up finding the actual global optimizer; this does not hold true for the ALM. We remark that we verified that, as expected, the computing time is indeed entirely dominated by projection steps. 

Then, we turn to the experiments on an instance where the number of nonlinear constraints shared by all components of the feasible set are numerous and dominate the complexity of solving each subproblem in the enumeration approach. In particular, we consider the previous problem with $N=50$, $n=5$, $m=80$, $s=7$. Here we set $\tau_0=0.1$, $\alpha_\tau=1.5$, $\epsilon_\text{in}=0.02$ for PD and $\tau_0=1$, $\alpha_\tau=1.5$ $m=4$, $\sigma=0.05$, $\epsilon_\text{in}=0.1$ for the ALM. The experiment is repeated 20 times for different random seeds. The results are reported in the form of performance profiles in Figure \ref{fig:diff_dis}. 

\begin{figure}[ht]
	\centering
	\includegraphics[width=0.8\textwidth]{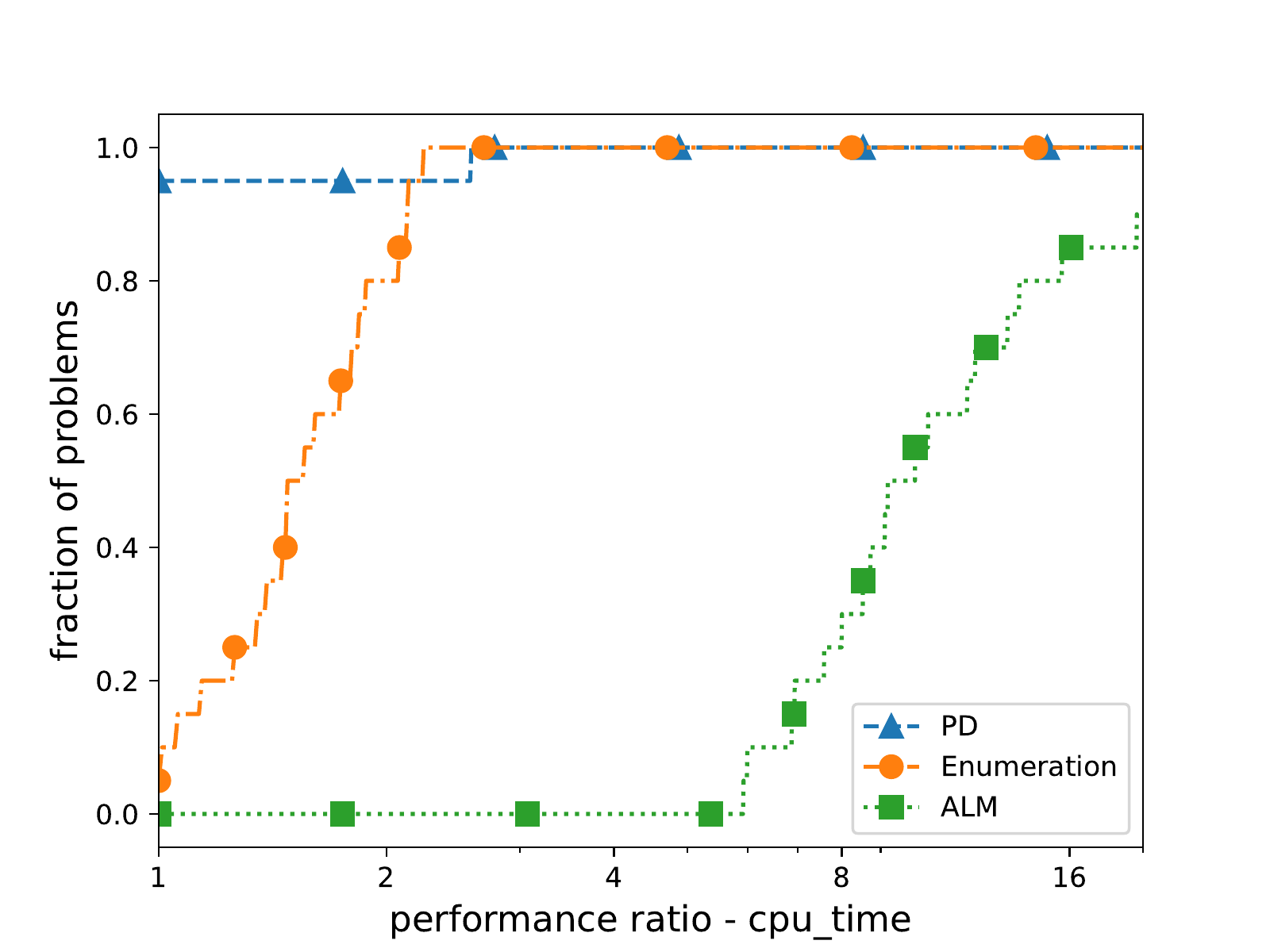}
	\caption{Performance profiles of runtime attained by PD, ALM and the enumeration strategy (with SLSQP) on 20 disjoint programming problems. When a solver does not attain the global minimum, the corresponding runtime is considered infinite when building the profile.}
	\label{fig:diff_dis}
\end{figure}

We deduce that Penalty Decomposition can indeed be a good choice in particularly complicated settings: the global optimizer was always reached, and this result was obtained in a consistently more efficient way than the brute force approach; the ALM does not have a comparable appeal in this context, the reason arguably being the much higher frequency of it resorting to the projection operation.

\section{Conclusions}\label{Sec:Conclusions}

The current paper considers a penalty decomposition scheme
for optimization problems with geometric constraints. It
generalizes existing penalty decomposition schemes both 
by taking advantage of a general abstract (and
usually complicated) constraint (as opposed to having only
particular instances like cardinality constraints) and by 
including further (though supposingly simple) constraints.
The idea and the convergence theory of this method is also
related to recent augmented Lagrangian techniques, but the
decomposition idea turns out to be numerically superior by
allowing more efficient subproblem solvers and using
many less projection steps.

In principle, it should be possible to extend the decomposition
idea to the class of (safeguarded) augmented Lagrangian methods.
Another, and related, question is whether one can exploit 
additional properties of augmented Lagrangian methods in order
to improve the existing convergence theory. For example,
augmented Lagrangian techniques have very strong convergence
properties in the convex case. The particular classes of problems
discussed in this paper are nonconvex, but the nonconvexity
mainly arises from the fact that the abstract set $ D $ is
nonconvex. Since we deal with the complicated set $ D $
explicitly, so that all iterates are feasible with respect to
this set, a natural question is therefore whether one can
prove stronger convergence properties in those situations where
the remaining functions and constraints are convex. This will
be part of our future research.



\begin{thebibliography}{10}
	\expandafter\ifx\csname url\endcsname\relax
	\def\url#1{\texttt{#1}}\fi
	\expandafter\ifx\csname doi\endcsname\relax
	\def\doi#1{\burlalt{doi:#1}{http://dx.doi.org/#1}}\fi
	\expandafter\ifx\csname urlprefix\endcsname\relax\def\urlprefix{URL }\fi
	\expandafter\ifx\csname href\endcsname\relax
	\def\href#1#2{#2}\fi
	\expandafter\ifx\csname burlalt\endcsname\relax
	\def\burlalt#1#2{\href{#2}{#1}}\fi
	
	\bibitem{Achtziger2008}
	W.~Achtziger and C.~Kanzow.
	\newblock Mathematical programs with vanishing constraints: optimality
	conditions and constraint qualifications.
	\newblock {\em Mathematical Programming}, 114(1):69--99, 2008.
	
	\bibitem{Andreani2011}
	R.~Andreani, G.~Haeser, and J.~M. Mart{\'\i}nez.
	\newblock On sequential optimality conditions for smooth constrained
	optimization.
	\newblock {\em Optimization}, 60(5):627--641, 2011.
	
	\bibitem{Andreani2016}
	R.~Andreani, J.~M. Martinez, A.~Ramos, and P.~J. Silva.
	\newblock A cone-continuity constraint qualification and algorithmic
	consequences.
	\newblock {\em SIAM Journal on Optimization}, 26(1):96--110, 2016.
	
	\bibitem{Andreani2018}
	R.~Andreani, J.~M. Martinez, A.~Ramos, and P.~J. Silva.
	\newblock Strict constraint qualifications and sequential optimality conditions
	for constrained optimization.
	\newblock {\em Mathematics of Operations Research}, 43(3):693--717, 2018.
	
	\bibitem{BauschkeCombettes2011}
	H.~H. Bauschke and P.~L. Combettes.
	\newblock {\em Convex Analysis and Monotone Operator Theory in Hilbert Spaces}.
	\newblock Springer, 2011.
	\newblock \doi{10.1007/978-1-4419-9467-7}.
	
	\bibitem{BeckEldar2013}
	A.~Beck and Y.~C. Eldar.
	\newblock Sparsity constrained nonlinear optimization: optimality conditions
	and algorithms.
	\newblock {\em SIAM Journal on Optimization}, 23(3):1480--1509, 2013.
	\newblock \doi{10.1137/120869778}.
	
	\bibitem{Ben2001}
	A.~Ben-Tal and A.~Nemirovski.
	\newblock {\em Lectures on Modern Convex Optimization: Analysis, Algorithms,
		and Engineering Applications}.
	\newblock SIAM, 2001.
	
	\bibitem{Benko2022}
	M.~Benko, M.~{\v{C}}ervinka, and T.~Hoheisel.
	\newblock Sufficient conditions for metric subregularity of constraint systems
	with applications to disjunctive and ortho-disjunctive programs.
	\newblock {\em Set-Valued and Variational Analysis}, 30(1):143--177, 2022.
	
	\bibitem{Benko2018}
	M.~Benko and H.~Gfrerer.
	\newblock New verifiable stationarity concepts for a class of mathematical
	programs with disjunctive constraints.
	\newblock {\em Optimization}, 67(1):1--23, 2018.
	
	\bibitem{Bertsekas1997}
	D.~P. Bertsekas.
	\newblock {\em Nonlinear {P}rogramming}.
	\newblock Athena Scientific, 1999.
	
	\bibitem{Bertsimas2022}
	D.~Bertsimas and R.~Cory-Wright.
	\newblock A scalable algorithm for sparse portfolio selection.
	\newblock {\em INFORMS Journal on Computing}, 2022.
	
	\bibitem{Boergens2019}
	E.~B\"orgens, C.~Kanzow, and D.~Steck.
	\newblock Local and global analysis of multiplier methods for constrained
	optimization in {B}anach spaces.
	\newblock {\em SIAM Journal on Control and Optimization}, 57(6):3694--3722,
	2019.
	
	\bibitem{Burdakov2016}
	O.~P. Burdakov, C.~Kanzow, and A.~Schwartz.
	\newblock Mathematical programs with cardinality constraints: reformulation by
	complementarity-type conditions and a regularization method.
	\newblock {\em SIAM Journal on Optimization}, 26(1):397--425, 2016.
	
	\bibitem{Burer2002}
	S.~Burer, R.~D. Monteiro, and Y.~Zhang.
	\newblock Maximum stable set formulations and heuristics based on continuous
	optimization.
	\newblock {\em Mathematical Programming}, 94(1):137--166, 2002.
	
	\bibitem{Candes2009}
	E.~J. Cand{\`e}s and B.~Recht.
	\newblock Exact matrix completion via convex optimization.
	\newblock {\em Foundations of Computational Mathematics}, 9(6):717--772, 2009.
	
	\bibitem{Cocchi2020}
	G.~Cocchi, T.~Levato, G.~Liuzzi, and M.~Sciandrone.
	\newblock A concave optimization-based approach for sparse multiobjective
	programming.
	\newblock {\em Optimization Letters}, 14(3):535--556, 2020.
	
	\bibitem{Dolan2002}
	E.~D. Dolan and J.~J. Mor{\'e}.
	\newblock Benchmarking optimization software with performance profiles.
	\newblock {\em Mathematical Programming}, 91(2):201--213, 2002.
	
	\bibitem{Flegel2007}
	M.~L. Flegel, C.~Kanzow, and J.~V. Outrata.
	\newblock Optimality conditions for disjunctive programs with application to
	mathematical programs with equilibrium constraints.
	\newblock {\em Set-Valued Analysis}, 15(2):139--162, 2007.
	
	\bibitem{Galvan2020}
	G.~Galvan, M.~Lapucci, T.~Levato, and M.~Sciandrone.
	\newblock An alternating augmented {L}agrangian method for constrained
	nonconvex optimization.
	\newblock {\em Optimization Methods and Software}, 35(3):502--520, 2020.
	
	\bibitem{Grippo2000}
	L.~Grippo and M.~Sciandrone.
	\newblock On the convergence of the block nonlinear {G}auss-{S}eidel method
	under convex constraints.
	\newblock {\em Operations Research Letters}, 26(3):127--136, 2000.
	
	\bibitem{Guignard1987}
	M.~Guignard and S.~Kim.
	\newblock Lagrangean decomposition: A model yielding stronger {L}agrangean
	bounds.
	\newblock {\em Mathematical Programming}, 39(2):215--228, 1987.
	
	\bibitem{Gurobi}
	{Gurobi Optimization, LLC}.
	\newblock {Gurobi Optimizer Reference Manual}, 2022.
	\newblock \urlprefix\url{https://www.gurobi.com}.
	
	\bibitem{Hastie2009}
	T.~Hastie, R.~Tibshirani, and J.~Friedman.
	\newblock {\em The Elements of Statistical Learning: Data Mining, Inference,
		and Prediction}.
	\newblock Springer, second edition, 2009.
	
	\bibitem{Hosseini2019}
	S.~Hosseini, D.~R. Luke, and A.~Uschmajew.
	\newblock Tangent and normal cones for low-rank matrices.
	\newblock In {\em Nonsmooth Optimization and Its Applications}, pages 45--53.
	Springer, 2019.
	
	\bibitem{JiaKanzowMehlitzWachsmuth2021}
	X.~Jia, C.~Kanzow, P.~Mehlitz, and G.~Wachsmuth.
	\newblock An augmented {L}agrangian method for optimization problems with
	structured geometric constraints.
	\newblock {\em Mathematical Programming}, 2022.
	\newblock \doi{10.1007/s10107-022-01870-z}.
	
	\bibitem{Jornsten1985}
	K.~O. J{\"o}rnsten, M.~N{\"a}sberg, and P.~A. Smeds.
	\newblock {\em Variable splitting: A new Lagrangean relaxation approach to some
		mathematical programming models}.
	\newblock Universitetet i Link{\"o}ping/Tekniska H{\"o}gskolan i Link{\"o}ping.
	Department of~…, 1985.
	
	\bibitem{Kanzow2017}
	C.~Kanzow and D.~Steck.
	\newblock An example comparing the standard and safeguarded augmented
	{L}agrangian methods.
	\newblock {\em Operations Research Letters}, 45(6):598--603, 2017.
	
	\bibitem{Kishore2017}
	N.~Kishore~Kumar and J.~Schneider.
	\newblock Literature survey on low rank approximation of matrices.
	\newblock {\em Linear and Multilinear Algebra}, 65(11):2212--2244, 2017.
	
	\bibitem{Lammel2022}
	S.~L{\"a}mmel and V.~Shikhman.
	\newblock On nondegenerate {M}-stationary points for sparsity constrained
	nonlinear optimization.
	\newblock {\em Journal of Global Optimization}, 82(2):219--242, 2022.
	
	\bibitem{Lapucci2022}
	M.~Lapucci.
	\newblock {\em Theory and Algorithms for Sparsity Constrained Optimization
		Problems}.
	\newblock PhD thesis, University of Florence, Italy, 2022.
	
	\bibitem{Lapucci2021}
	M.~Lapucci, T.~Levato, and M.~Sciandrone.
	\newblock Convergent inexact penalty decomposition methods for
	cardinality-constrained problems.
	\newblock {\em Journal of Optimization Theory and Applications},
	188(2):473--496, 2021.
	
	\bibitem{Liu1989}
	D.~C. Liu and J.~Nocedal.
	\newblock On the limited memory {BFGS} method for large scale optimization.
	\newblock {\em Mathematical Programming}, 45(1):503--528, 1989.
	
	\bibitem{Lu2013}
	Z.~Lu and Y.~Zhang.
	\newblock Sparse approximation via penalty decomposition methods.
	\newblock {\em SIAM Journal on Optimization}, 23(4):2448--2478, 2013.
	
	\bibitem{Markovsky2012}
	I.~Markovsky.
	\newblock {\em Low Rank Approximation: Algorithms, Implementation,
		Applications}.
	\newblock Springer, 2012.
	
	\bibitem{Mehlitz2020c}
	P.~Mehlitz.
	\newblock Asymptotic stationarity and regularity for nonsmooth optimization
	problems.
	\newblock {\em Journal of Nonsmooth Analysis and Optimization}, 1, 2020.
	
	\bibitem{Mehlitz2020}
	P.~Mehlitz.
	\newblock On the linear independence constraint qualification in disjunctive
	programming.
	\newblock {\em Optimization}, 69(10):2241--2277, 2020.
	
	\bibitem{Mehlitz2020b}
	P.~Mehlitz.
	\newblock Stationarity conditions and constraint qualifications for
	mathematical programs with switching constraints.
	\newblock {\em Mathematical Programming}, 181(1):149--186, 2020.
	
	\bibitem{Mordukhovich2018}
	B.~S. Mordukhovich.
	\newblock {\em Variational Analysis and Applications}.
	\newblock Springer, 2018.
	\newblock \doi{10.1007/978-3-319-92775-6}.
	
	\bibitem{Recht2010}
	B.~Recht, M.~Fazel, and P.~A. Parrilo.
	\newblock Guaranteed minimum-rank solutions of linear matrix equations via
	nuclear norm minimization.
	\newblock {\em SIAM Review}, 52(3):471--501, 2010.
	
	\bibitem{RockafellarWets2009}
	R.~T. Rockafellar and R.~J.-B. Wets.
	\newblock {\em Variational Analysis}.
	\newblock Springer, 2009.
	\newblock \doi{10.1007/978-3-642-02431-3}.
	
	\bibitem{Xue2007}
	Y.~Xue, X.~Liao, L.~Carin, and B.~Krishnapuram.
	\newblock Multi-task learning for classification with {D}irichlet process
	priors.
	\newblock {\em Journal of Machine Learning Research}, 8(1), 2007.
	
	\bibitem{Ye1999}
	J.~Ye.
	\newblock Optimality conditions for optimization problems with complementarity
	constraints.
	\newblock {\em SIAM Journal on Optimization}, 9(2):374--387, 1999.
	
	\bibitem{Zhang2011}
	Y.~Zhang and Z.~Lu.
	\newblock Penalty decomposition methods for rank minimization.
	\newblock {\em Advances in Neural Information Processing Systems}, 24, 2011.
	
	\bibitem{Zhang2021}
	Y.~Zhang and Q.~Yang.
	\newblock A survey on multi-task learning.
	\newblock {\em IEEE Transactions on Knowledge and Data Engineering}, 2021.
	
\end{thebibliography}
\end{document}